\documentclass[12pt, a4paper, reqno]{amsart}

\usepackage[utf-8]{inputenc}
\usepackage[T1]{fontenc}
\usepackage{amsmath}
\usepackage{amssymb}
\usepackage{enumerate}
\usepackage[font=small, labelfont=sc]{caption}

\theoremstyle{plain}
\newtheorem{theo}{Theorem}
\newtheorem{prop}{Proposition}
\newtheorem{lemma}{Lemma}

\newcommand{\tLeg}[2]{\genfrac{(}{)}{}{1}{#1}{#2}}  

\newcommand{\abs}[1]{\lvert #1 \rvert}              
\newcommand{\conj}[1]{\bar{#1}}                     
\newcommand{\conjl}[1]{\overline{#1}}               
\newcommand{\field}[1]{\mathbb{#1}}                 
\newcommand{\Fq}[1][]{\field{F}_{\!q^{#1}}}         
\newcommand{\Fqa}[1][]{\field{F}_{\!q^{#1}}^\ast}   
\newcommand{\gen}[1]{\langle #1 \rangle}            
\newcommand{\num}[1]{\lvert #1 \rvert}              

\newcommand{\F}{\field{F}}                          
\newcommand{\imu}{\mathsf{i}}                       
\newcommand{\Nap}{\mathsf{e}}                       
\newcommand{\tc}{\text{,}}                          
\newcommand{\tdots}{\tc\ldots\tc\ }                 
\newcommand{\tp}{\text{.}}                          
\newcommand{\Z}{\field{Z}}                          

\newcommand{\cck}{\conjl{\chi_3}}                   
\newcommand{\ccn}{\conjl{\chi_4}}                   
\newcommand{\ck}{\chi_3}                            
\newcommand{\cn}{\chi_4}                            
\newcommand{\cpk}{\conjl{\pi_3}}                    
\newcommand{\cpn}{\conjl{\pi_4}}                    
\newcommand{\pk}{\pi_3}                             
\newcommand{\pn}{\pi_4}                             

\DeclareMathOperator{\ind}{ind}                     
\DeclareMathOperator{\Norm}{Norm}                   
\DeclareMathOperator{\ord}{ord}                     
\DeclareMathOperator{\Real}{Re}                     
\DeclareMathOperator{\Tr}{Tr}                       
\DeclareMathOperator{\tr}{tr}                       

\begin{document}

\title[Irreducible polynomials]{Irreducible polynomials with prescribed trace and restricted norm}

\author{K.~Kononen}
\address{K.~Kononen: Department of Mathematical Sciences, University of Oulu, P.O.\ Box 3000, FIN-90014 OULUN YLIOPISTO, Finland}
\email{kkononen@paju.oulu.fi}

\author{M.~Moisio}
\address{M.~Moisio: Department of Mathematics and Statistics, University of Vaasa, P.O.\ Box 700, FIN-65101 VAASA, Finland}
\email{mamo@uwasa.fi}

\author{M.~Rinta-aho}
\address{M.~Rinta-aho: Department of Mathematical Sciences, University of Oulu, P.O.\ Box 3000, FIN-90014 OULUN YLIOPISTO, Finland}
\email{marko.rinta-aho@oulu.fi}

\author{K.~V\"a\"an\"anen}
\address{K.~V\"a\"an\"anen: Department of Mathematical Sciences, University of Oulu, P.O.\ Box 3000, FIN-90014 OULUN YLIOPISTO, Finland}
\email{keijo.vaananen@oulu.fi}

\keywords{Irreducible polynomials; Monomial exponential sums; Gauss sums; Jacobi sums}

\begin{abstract}
Let $\Fq$, $q = p^r$, be a finite field with a primitive element $g$. 
In this paper we use exponential sums and Jacobi sums to compute the number of the irreducible polynomials of degree $m$ over $\Fq$ 
with trace fixed and norm restricted to a coset of a subgroup $\gen{g^s}$, $s \mid (q-1)$. 
We give the number explicitly for $s=2$, $3$, $4$ when $q=p$, and for $s \mid (p^e+1)$ when $r = 2en$. 
Finally, we give explicit formulae for the number when both trace and norm are fixed, $p=2$ and $m \le 30$. 
\end{abstract}

\maketitle

\section{introduction}

Let $p$ be a prime number, $\Fq$ a finite field with $q=p^r$ elements, and $g$ a primitive element of $\Fq$. 
The explicit enumeration of irreducible polynomials 
\[
  f(x) = x^m - ax^{m-1} + \cdots + (-1)^mb \in \Fq[][x]
\] 
with some preassigned coefficients fixed is quite hard problem in general and it has been tackled only in certain special cases. 
For example, Carlitz \cite{Carl} and Yucas \cite{Yucas} obtained explicit formulae for the number of $f(x)$ with $a$ or $b$ fixed. 
Carlitz also obtained explicit formulae for the number of $f(x)$ with $a$ fixed and $b$ in a fixed coset of the group of squares in $\Fqa$. 
Moisio \cite{MoiKlo} considered the enumeration problem when both $a$ and $b$ are fixed, and gave the number of $f(x)$ 
in terms of exponential sums and in terms of the number of rational points on certain algebraic varieties defined over $\Fq$. 
Especially, the number of irreducible cubic polynomials with $a$ and $b$ fixed was given in terms of cubic Gauss sums (case $a=0$) 
and in terms of the number of rational points on the elliptic curves over $\Fq$ defined by $\mathcal{E} : y^2 + ba^{-3} y + xy = x^3$, 
which, in a way, indicates the hardness of the explicit enumeration problem. 
For results on the enumeration problem when some other coefficients than $a$ and $b$ are fixed we refer to the survey by Cohen \cite{Cohen}, 
and to a recent work by Moisio and Ranto \cite{MoiRa}.

The aim of this paper is to generalize results of \cite{Carl} by giving explicit formulae for the number of $f(x)$ with $a$ fixed 
and $b$ in a fixed coset of a subgroup $\gen{g^s}$. 
Actually, we shall do this in the following three special cases: 

\begin{itemize}
\item
$s=2$ (Carlitz's case),

\item
$s=3$,

\item
$s=4$,

\item
$r=2en$ and $s\,(>1)$ is any factor of $p^e+1$.
\end{itemize}
Moreover, we shall give explicit formulae for the number of $f(x)$ with both $a$ and $b$ fixed under the assumptions $q = 2^r$ and $m \le 30$. 

The method used in this paper is essentially the one used in \cite{MoiKlo} but here explicit evaluation of Jacobi sums 
and certain exponential sums are used instead of the theory of algebraic varieties. 
We also note that our method is more elementary than the method used in \cite{Carl} in the sense that the use of L-functions is avoided.


\section{Notations and basic formulae} \label{sec:notbasf}

We fix the following notations.

\vspace{1ex}
\begin{center}
\begin{tabular}{c@{\hspace{1em}}p{0.63\textwidth}}
$p$, $q$, $r$, $s$, $h$, $m$, $t$ 
                     & positive integers, $p$ prime, $q = p^r$, $s \mid (q-1)$, $h \in \{ 0, 1, \ldots, s-1 \}$, $m \ge 2$, $t \mid m$ \\
     $d$, $l$        & $d = \gcd(\tfrac{m}{t}, s)$, $l = \gcd(t, \tfrac{s}{d})$                          \\
 $\Tr_t$, $\Norm_t$  & the trace and norm from $\Fq[t]$ onto $\Fq$                                       \\
$\gamma = \gamma_m$  & a fixed primitive element of $\Fq[m]$                                             \\
     $\gamma_t$      & the primitive element of $\Fq[t]$ that is the norm of $\gamma$ onto $\Fq[t]$      \\
         $g$         & $\Norm_m (\gamma_m)$, a primitive elment of $\Fq$; also $g = \gamma_1 = \Norm_t (\gamma_t)$ \\
   $P_m(a, s, h)$    & the number of the irreducible polynomials $f(x) = x^m - ax^{m-1} + \cdots + (-1)^m b \in \Fq[][x]$, 
                       where $a$ is fixed and $b \in g^h \gen{g^s} \subseteq \Fqa$                       \\
$S_t = S_t(a, s, h)$ & the set of $x$ in $\Fq[t]$ with $\Tr_m(x) = a$ and $\Norm_m(x) \in g^h \gen{g^s}$ \\
$T_t = T_t(a, s, h)$ & the set of $x$ in $S_t$ with $x \notin \Fq[k]$ for any $k < t$                    \\
         $N_t$       & the number of elements in $S_t$                                                   \\
         $e_t$       & the canonical additive character $e_t(x) = \Nap^{2\pi \imu \tr_t(x)/p}$ of $\Fq[t]$, 
                       where $\tr_t$ is the absolute trace $\Fq[t] \rightarrow \F_{\!p}$                 \\
\end{tabular}
\end{center}
\vspace{1ex}

Note that $N_m = \sum_{t \mid m} \num{T_t}$ and the M\"obius inversion gives, see \cite[Lemma 1]{MoiKlo},
\begin{align} \label{eq:PmsumNt}
  P_m(a, s, h) = \frac{1}{m} \sum_{t \mid m} \mu(\tfrac{m}{t}) N_t \tp
\end{align}
Thus the knowledge of $N_t$ for $t \mid m$ gives $P_m(a, s, h)$, and therefore we consider $N_t$.

From the definition of $S_t$ it follows that
\[
  x \in S_t \iff \tfrac{m}{t} \Tr_t(x) = a \text{ and } \Norm_t (x^{\frac{m}{t}}) \in g^h \gen{g^s} \tp
\]
By denoting $x = \gamma_t^i$, $i \in \{ 0, 1, \ldots, q^t - 2 \}$, we see that the condition $\Norm_t (x^{m/t}) \in g^h \gen{g^s}$ is satisfied 
if and only if the congruence
\begin{align} \label{eq:icong}
  \frac{m}{t} i \equiv h \pmod s
\end{align}
holds. 
If $d \nmid h$ then \eqref{eq:icong} has no solution, and if $d \mid h$ then \eqref{eq:icong} has solutions
\begin{align} \label{eq:isols}
  i = i_0 + j \frac{s}{d} \tc \qquad j = 0 \tc\ 1 \tdots \frac{d}{s} (q^t - 1) - 1 \tc
\end{align}
where $i_0$ is the solution of
\begin{align} \label{eq:i0def}
  \frac{m}{dt} i_0 \equiv \frac{h}{d} \quad \bigl( \bmod \, \frac{s}{d} \bigr) \tc \qquad 0 \le i_0 < \frac{s}{d} \tp
\end{align}
Thus we obtain
\begin{lemma} \label{le:Ntspecial}
\begin{enumerate}[(i)]
\item
$N_t = 0$ if $d \nmid h$.

\item
$N_t = 0$ if $p \mid \tfrac{m}{t}$ and $a \ne 0$.

\item
$N_t = \tfrac{d}{s} (q^t - 1)$ if $p \mid \tfrac{m}{t}$, $d \mid h$ and $a=0$.
\end{enumerate}
\end{lemma}
In the remaining cases
\begin{align} \label{eq:restpd}
  p \nmid \tfrac{m}{t} \text{ and } d \mid h \tp
\end{align}
To state a formula for $N_t$ in this case we use the canonical additive character $e_t$.

\begin{lemma} \label{le:NtsumMt}
If (\ref{eq:restpd}) holds then
\[
  N_t = \frac{d}{sq} (q^t - 1 + M_t) \tc
\]
where
\begin{align} \label{eq:Mtdef}
  M_t = \sum_{c \in \Fqa} e_1 (-\tfrac{t}{m} ca) \sum_{x \in \Fqa[t]} e_t (c \gamma_t^{i_0} x^\frac{s}{d}) \tp
\end{align}
\end{lemma}

\begin{proof}
By the definition of $S_t$, equation \eqref{eq:isols} and the orthogonality of characters we obtain
\begin{align*}
  qN_t &= \sum_{j=0}^{\frac{d}{s}(q^t-1) - 1} \sum_{c \in \Fq} e_1 \bigl( c (\Tr_t(\gamma_t^{i_0 + j\frac{s}{d}}) - \tfrac{t}{m} a) \bigr) \\
       &= \sum_{c \in \Fq} e_1(-\tfrac{t}{m} ca) \sum_{j=0}^{\frac{d}{s}(q^t-1) - 1} e_1 \bigl( \Tr_t (c \gamma_t^{i_0 + j\frac{s}{d}}) \bigr) 
          \displaybreak[1]\\
       &= \sum_{c \in \Fq} e_1(-\tfrac{t}{m} ca) \sum_{j=0}^{\frac{d}{s}(q^t-1) - 1} e_t (c \gamma_t^{i_0 + j\frac{s}{d}}) \\
       &= \frac{d}{s} \sum_{c \in \Fq} e_1(-\tfrac{t}{m} ca) \sum_{x \in \Fqa[t]} e_t (c \gamma_t^{i_0} x^\frac{s}{d}) \tp
\end{align*}
This proves Lemma \ref{le:NtsumMt}.
\end{proof}

We shall now consider separately the cases $a=0$ and $a \ne 0$. 
For this consideration let $n \mid (q-1)$ and let $H_n$ denote the subgroup of order $n$ of the multiplicative character group of $\Fq$ 
and $H_n^\ast = H_n \setminus \{ \lambda_0 \}$, where $\lambda_0$ is the trivial multiplicative character of $\Fq$. 
If $\lambda$ is a multiplicative character of $\Fq$, then $\lambda \circ \Norm_t$ is a multiplicative character of $\Fq[t]$. 
We define the Gauss sum $G_t(\lambda)$ over $\Fq[t]$ as follows:
\[
  G_t(\lambda) = \sum_{x \in \Fqa[t]} e_t(x) (\lambda \circ \Norm_t)(x) \tp
\]
The following Lemma 4 of \cite{MoiKlo} gives a connection of these sums with monomial exponential sums.

\begin{lemma} \label{le:monGcon}
Let $n$ be a positive factor of $q-1$ and let $\alpha \in \Fqa[t]$. 
Then
\[
  \sum_{x \in \Fqa[t]} e_t(\alpha x^n) = \sum_{\lambda \in H_n} G_t(\conj{\lambda}) (\lambda \circ \Norm_t) (\alpha) \tc
\]
where $\conj{\lambda} = \lambda^{-1}$.
\end{lemma}


\section{Case $a=0$}

Assume in this section that $a=0$. 
Then $M_t$ in \eqref{eq:Mtdef} has the following expressions.

\begin{lemma} \label{le:Mtbasrep}
Assume that (\ref{eq:restpd}) holds and $a=0$. 
Then
\[
  M_t = (q-1) \sum_{x \in \Fqa[t]} e_t(\gamma_t^{i_0} x^l) = (q-1) \sum_{\lambda \in H_l} G_t(\conj{\lambda}) \lambda(g^{i_0}) \tc
\]
where $l = \gcd (t, \tfrac{s}{d})$ is defined at the beginning of Section \ref{sec:notbasf}.
\end{lemma}

\begin{proof}
The claimed formulae for $M_t$ are equal by Lemma \ref{le:monGcon} (recall that $\Norm_t (\gamma_t) = g$). 
Substituting $a=0$ into \eqref{eq:Mtdef} we get by Lemma \ref{le:monGcon}
\begin{align}
  M_t &= \sum_{c \in \Fqa} \sum_{x \in \Fqa[t]} e_t (c \gamma_t^{i_0} x^\frac{s}{d}) 
         = \sum_{c \in \Fqa} \sum_{\lambda \in H_{s/d}} G_t(\conj{\lambda}) \lambda(\Norm_t (c \gamma_t^{i_0})) \nonumber\\
      &= \sum_{\lambda \in H_{s/d}} G_t(\conj{\lambda}) \lambda(g^{i_0}) \sum_{c \in \Fqa} \lambda(c^t) \tp \label{eq:Mtintsum}
\end{align}
Setting $n = \gcd(q-1, t)$ we have
\[
  \sum_{c \in \Fqa} \lambda(c^t) = \sum_{c \in \Fqa} \lambda^n (c) =
  \begin{cases}
    0   &\text{if $\lambda^n \ne \lambda_0$,} \\
    q-1 &\text{if $\lambda^n = \lambda_0$.}
  \end{cases}
\]
Since $\lambda \in H_{s/d}$ in \eqref{eq:Mtintsum}, the condition $\lambda^n = \lambda_0$ is equivalent 
to $\lambda \in H_n \cap H_{s/d} = H_l$, and the lemma follows. 
\end{proof}

Furthermore, we present $M_t$ in terms of (a special class of) Jacobi sums
\[
  J_t(\lambda) = \sum_{\substack{x_1\tc\ldots\tc x_t \in \Fq \\ x_1 + \cdots + x_t =1}} \lambda(x_1 \cdots x_t) \tc
\]
where $\lambda$ is a multiplicative character of $\Fq$ and, as usual, we define $\lambda(0) = 0$, if $\lambda \ne \lambda_0$, 
and $\lambda_0 (0) = 1$.

\begin{lemma} \label{le:MtJacobi0}
Assume that (\ref{eq:restpd}) holds and $a=0$. 
Then
\[
  M_t = (q-1) \Bigl( -1 + (-1)^t q \sum_{\lambda \in H_l^\ast} J_t(\lambda) \conj{\lambda} (g^{i_0}) \Bigr) \tc
\]
where $l = \gcd (t, \tfrac{s}{d})$.
\end{lemma}

\begin{proof}
In Lemma \ref{le:Mtbasrep} $G_t(\conjl{\lambda_0}) \lambda_0 (g^{i_0}) = -1$. 
For $\lambda \ne \lambda_0$, the Davenport-Hasse identity (see e.g.\ \cite[Theorem 5.14]{LidNi}) gives 
$G_t(\lambda) = (-1)^{t-1} G_1(\lambda)^t$ and \cite[Theorem 10.3.1]{BEW} gives $G_1(\lambda)^t = -q J_t(\lambda)$ since $l \mid t$. 
\end{proof}

As we shall see, Lemmas \ref{le:Mtbasrep} and \ref{le:MtJacobi0} give $M_t$ explicitly in many cases.


\section{Case $a \ne 0$}

The result corresponding to Lemmas \ref{le:Mtbasrep} and \ref{le:MtJacobi0} is for $a \ne 0$ the following lemma.

\begin{lemma} \label{le:MtJacobine0}
Assume that (\ref{eq:restpd}) holds and $a \ne 0$. 
Then
\begin{align*}
  M_t &= \sum_{\lambda \in H_{s/d}} G_t(\conj{\lambda}) G_1(\lambda^t) \lambda (a_0^t g^{i_0}) \\
      &= 1 + (-1)^{t-1} q \sum_{\lambda \in H_{s/d}^\ast} J_t(\conj{\lambda}) \lambda ((-a_0)^t g^{i_0}) \tc
\end{align*}
where $a_0 = -\tfrac{m}{ta}$.
\end{lemma}

\begin{proof}
Substituting $c \mapsto -\tfrac{mc}{ta} = a_0 c$ into \eqref{eq:Mtdef} we get by Lemma \ref{le:monGcon} and the Davenport-Hasse identity
\begin{align*}
  M_t &= \sum_{c \in \Fqa} e_1 (c) \sum_{\lambda \in H_{s/d}} G_t(\conj{\lambda}) \lambda(\Norm_t(a_0 c \gamma_t^{i_0})) \\
      &= \sum_{c \in \Fqa} e_1 (c) \sum_{\lambda \in H_{s/d}} G_t(\conj{\lambda}) \lambda(a_0^t c^t g^{i_0})             \\
      &= \sum_{\lambda \in H_{s/d}} G_t(\conj{\lambda}) \lambda(a_0^t g^{i_0}) \sum_{c \in \Fqa} e_1 (c) \lambda^t (c)   \\
      &= (-1)^{t-1} \sum_{\lambda \in H_{s/d}} G_1(\conj{\lambda})^t G_1(\lambda^t) \lambda(a_0^t g^{i_0})               \\
      &= 1 + (-1)^{t-1} \sum_{\lambda \in H_{s/d}^\ast} G_1(\conj{\lambda})^t G_1(\lambda^t) \lambda(a_0^t g^{i_0}) \tp 
\end{align*}
Thus
\begin{align}
\begin{split} \label{eq:MtGsplit}
  (-1)^{t-1} (M_t - 1) &= \sum_{\lambda \in H_l^\ast} G_1(\conj{\lambda})^t G_1(\lambda^t) \lambda(a_0^t g^{i_0}) \\
  &\quad+ \sum_{\lambda \in H_{s/d}^\ast \setminus H_l^\ast} G_1(\conj{\lambda})^t G_1(\lambda^t) \lambda(a_0^t g^{i_0}) \tp
\end{split}
\end{align}
By Theorems 10.3.1 and 1.1.4~(b) of \cite{BEW}, 
\[
  G_1(\conj{\lambda})^t =
  \begin{cases}
    -q J_t(\conj{\lambda}) &\text{if $\lambda^t = \lambda_0$,} \\
    \lambda((-1)^t) J_t(\conj{\lambda}) \conjl{G_1(\lambda^t)} &\text{if $\lambda^t \ne \lambda_0$.}
  \end{cases}
\]
For $\lambda \in H_l \subseteq H_t$, $G_1(\lambda^t) = -1$ and $\lambda((-1)^t) = 1$. 
Hence in \eqref{eq:MtGsplit}
\begin{align} \label{eq:inHlpartJ}
  \sum_{\lambda \in H_l^\ast} G_1(\conj{\lambda})^t G_1(\lambda^t) \lambda(a_0^t g^{i_0}) 
    = q \sum_{\lambda \in H_l^\ast} J_t(\conj{\lambda}) \lambda((-a_0)^t g^{i_0}) \tp
\end{align}
For $\lambda \in H_{s/d}^\ast \setminus H_l^\ast$, $\conjl{G_1(\lambda^t)} G_1(\lambda^t) = \abs{G_1(\lambda^t)}^2 = q$ and in \eqref{eq:MtGsplit}
\begin{align*}
  \sum_{\lambda \in H_{s/d}^\ast \setminus H_l^\ast} &G_1(\conj{\lambda})^t G_1(\lambda^t) \lambda(a_0^t g^{i_0}) 
    = q \sum_{\lambda \in H_{s/d}^\ast \setminus H_l^\ast} J_t(\conj{\lambda}) \lambda((-a_0)^t g^{i_0}) \\
    &= q \Bigl( \sum_{\lambda \in H_{s/d}^\ast} J_t(\conj{\lambda}) \lambda((-a_0)^t g^{i_0}) 
      - \sum_{\lambda \in H_l^\ast} J_t(\conj{\lambda}) \lambda((-a_0)^t g^{i_0}) \Bigr) \tp
\end{align*}
Combining this with \eqref{eq:MtGsplit} and \eqref{eq:inHlpartJ} we obtain the lemma.
\end{proof}

In some cases we are able to compute monomial sums $\sum_{x \in \Fqa[t]} e_t(\alpha x^n)$ explicitly. 
In such cases Lemma \ref{le:Mtbasrep} is useful for $a=0$. 
The following lemma gives similar formula for $a \ne 0$.

\begin{lemma} \label{le:Mtmonsum}
Assume that (\ref{eq:restpd}) holds and $a \ne 0$. 
Then
\[
  M_t = \frac{1}{u} \sum_{j=0}^{u-1} \Bigl( \sum_{x \in \Fqa[t]} e_t \bigl( a_0 \gamma_t^{t_0 j + i_0} x^\frac{s}{d} \bigr) \Bigr) 
    \Bigl( \sum_{c \in \Fqa} e_1(g^j c^u) \Bigr) \tc
\]
where $a_0 = -\tfrac{m}{ta}$, $t_0 = \tfrac{q^t-1}{q-1}$, and $u = \tfrac{s}{dl}$ 
with $l = \gcd(t, \tfrac{s}{d}) = \gcd(t_0, \tfrac{s}{d})$.
\end{lemma}

\begin{proof}
First we observe that $t_0 = (q-1)(q^{t-2} + 2q^{t-3} + \cdots + (t-2)q + t - 1) + t$ and therefore $l = \gcd(t_0, \frac{s}{d})$. 

Substituting $c \mapsto a_0 c$ and noting that $g = \gamma_t^{t_0}$, \eqref{eq:Mtdef} transforms into
\[
  M_t = \sum_{c \in \Fqa} e_1(c)\sum_{x \in \Fqa[t]} e_t (a_0 c \gamma_t^{i_0} x^{\frac{s}{d}})
      = \sum_{i=0}^{q-2} e_1(\gamma_t^{t_0 i}) \sum_{x \in \Fqa[t]} e_t (a_0 \gamma_t^{t_0 i+i_0} x^{\frac{s}{d}}) \tp
\]  
By the partitition $\gen{\gamma_t^{t_0}} = \bigcup_{j=0}^{u-1} \gamma_t^{t_0j} \gen{\gamma_t^{t_0 u}}$ each element 
in $\gen{\gamma_t^{t_0}}$ can be written in the form $\gamma_t^{t_0 j} \gamma_t^{t_0 uk}$ with $j \in \{ 0, \dots, u-1 \}$ 
and $k \in \{ 0, \dots, \frac{q-1}{u} - 1 \}$. 
Thus,
\begin{align*}
  \sum_{x \in \Fqa[t]} e_t(a_0 \gamma_t^{t_0 i + i_0} x^{\frac{s}{d}})
    &= \sum_{x \in \Fqa[t]} e_t(a_0 \gamma_t^{t_0 j + i_0} (\gamma_t^{k t_0/l} x)^{\frac{s}{d}}) \\
    &= \sum_{x \in \Fqa[t]} e_t(a_0 \gamma_t^{t_0 j + i_0} x^{\frac sd}) \tc
\end{align*}
and consequently
\[
  M_t = \sum_{j=0}^{u-1} \sum_{x \in \Fqa[t]} e_t(a_0 \gamma_t^{t_0 j + i_0} x^{\frac{s}{d}})
    \sum_{k=0}^{\frac{q-1}{u}-1} e_1(\gamma_t^{t_0 j} \gamma_t^{t_0 uk}) \tp
\]
Here the inner sum equals
\[
  \frac{1}{u} \sum_{k=0}^{q-2} e_1(g^j g^{ku}) = \frac{1}{u} \sum_{c \in \Fqa} e_1(g^j c^u) \tc
\]
and the proof is complete.  
\end{proof}


\section{Number of polynomials in certain special cases}

In this section we consider some special cases when $M_t$, and hence $P_m(a, s, h)$, can be given expicitly. 
One case is that $s$ is small. 
Then $\lambda$ in the summations of Lemmas \ref{le:MtJacobi0} and \ref{le:MtJacobine0} has small order. 
The Jacobi sums for characters of several small orders have been computed explicitly in \cite{BEW}.

Another classes when $M_t$ can be computed explicitly (or up to two choices) are the semiprimitive and index $2$ cases for $p=2$ 
(see Subsection \ref{sbs:sprind2}, p.~\pageref{sbs:sprind2}, for definitions). 
In these cases the monomial sums, or at least their value distribution, in Lemmas \ref{le:Mtbasrep} and \ref{le:Mtmonsum} can be evaluated.

We shall consider several small $s$ and semiprimitive and index $2$ cases in the following subsections.


\subsection{Case $s=2$ (Carlitz's case)}

The case $s=2$ was studied already by Carlitz in \cite{Carl}. 
Now $b \in g^h \gen{g^2}$ and $h = 0$ or $1$ according to whether $b$ is a square or a non-square in $\Fqa$. 
In addition, $p$ must be odd since $2 \mid (q-1)$. 
We have now three possibilities for $d$ and $l$:
\[
  (d, l) =
  \begin{cases}
    (1, 1) &\text{if $2 \nmid \tfrac{m}{t}$, $2 \nmid t$,} \\
    (1, 2) &\text{if $2 \nmid \tfrac{m}{t}$, $2 \mid t$,}  \\
    (2, 1) &\text{if $2 \mid \tfrac{m}{t}$.}
  \end{cases}
\]
Let us now compute $M_t$ and $N_t$ assuming \eqref{eq:restpd}. 
For other cases, $N_t$ can be computed with Lemma \ref{le:Ntspecial}. 
After computing $M_t$ the $N_t$ is obtained from Lemma \ref{le:NtsumMt}, see Theorem \ref{th:Nts=2} below.

If $(d, l) = (1, 1)$ then $i_0 = h$ in \eqref{eq:i0def}. 
For $a=0$ we have by Lemma \ref{le:Mtbasrep}
\[
  M_t = (q-1) \sum_{x \in \Fqa[t]} e_t(\gamma_t^h x) = 1-q \tp
\]
For $a \ne 0$, let $\rho$ be the multiplicative character of order $2$ of $\Fq$. 
Then $\conj{\rho} = \rho$ and $\rho(g^h) = (-1)^h$. 
Further,
\begin{align} \label{eq:Jtrho2}
  J_t(\rho) =
  \begin{cases}
    -\rho((-1)^\frac{t}{2}) q^\frac{t-2}{2}  &\text{if $t$ is even,} \\
    \rho((-1)^\frac{t-1}{2}) q^\frac{t-1}{2} &\text{if $t$ is odd,}  \\
  \end{cases}
\end{align}
by \cite[Theorem 10.2.2]{BEW}. 
As now $t$ is odd, Lemma \ref{le:MtJacobine0} and \eqref{eq:Jtrho2} give
\begin{align*}
  M_t &= 1 + q J_t(\rho) \rho((-a_0)^t g^h) 
      = 1 + q \rho((-1)^\frac{t-1}{2}) q^\frac{t-1}{2} \rho ((\tfrac{m}{ta})^t g^h) \\
      &= 1 + (-1)^h q^\frac{t+1}{2} \rho \bigl( (-1)^\frac{t-1}{2} \tfrac{m}{ta} \bigr) \tp
\end{align*}

If $(d, l) = (1, 2)$ then again $i_0 = h$ in \eqref{eq:i0def}. 
Now $t$ is even, so Lemma \ref{le:MtJacobi0} and \eqref{eq:Jtrho2} give for $a=0$
\begin{align*}
  M_t &= (q-1) (-1 + q J_t(\rho) \rho(g^h)) \\
      &= (q-1) \bigl( -1 - q \rho((-1)^\frac{t}{2}) q^\frac{t-2}{2} (-1)^h \bigr) \\
      &= (q-1) \bigl( -1 - (-1)^h q^\frac{t}{2} \rho((-1)^\frac{t}{2}) \bigr) \tp
\end{align*}

For $a \ne 0$ we get by Lemma \ref{le:MtJacobine0} and \eqref{eq:Jtrho2}
\begin{align*}
  M_t &= 1 - q J_t(\rho) \rho((-a_0)^t g^h) = 1 + q \rho((-1)^\frac{t}{2}) q^\frac{t-2}{2} (-1)^h \\
      &= 1 + (-1)^h \rho((-1)^\frac{t}{2}) q^\frac{t}{2} \tp
\end{align*}

If $(d, l) = (2, 1)$ then \eqref{eq:restpd} can hold only if $h=0$ ($N_t = 0$ for $h=1$ by Lemma \ref{le:Ntspecial}). 
By Lemmas \ref{le:Mtbasrep} and \ref{le:MtJacobine0}, $M_t = 1-q$ for $a=0$ and $M_t = 1$ for $a \ne 0$. 
Lemma \ref{le:NtsumMt} now gives the following theorem. 

\begin{theo} \label{th:Nts=2}
The values of $N_t$ for $s=2$ and assuming (\ref{eq:restpd}) are those listed in Table \ref{tbl:Nts2}.
\end{theo}

\begin{table}[!htb]
\centering
\caption{Values of $N_t$ for $s=2$ assuming \eqref{eq:restpd}.}
\label{tbl:Nts2}
\begin{tabular}{ccc}
\hline
 $a$  &                                             $N_t$                                              & $(d, l)$ \\
\hline
$a=0$ &                         \rule{0pt}{2.5ex}$\tfrac{1}{2} (q^{t-1} - 1)$                          & $(1, 1)$ \\
      & $\tfrac{1}{2} \bigl( q^{t-1} - 1 - (q-1) (-1)^h q^\frac{t-2}{2} \rho((-1)^\frac{t}{2}) \bigr)$ & $(1, 2)$ \\
      & $q^{t-1} - 1$ & $(2, 1)$ \\
\hline
$a \ne 0$ & \rule{0pt}{2.5ex}$\tfrac{1}{2} \bigl( q^{t-1} + (-1)^h q^\frac{t-1}{2} \rho \bigl( (-1)^\frac{t-1}{2} \tfrac{m}{ta} \bigr) \bigr)$ 
            & $(1, 1)$ \\
          & $\tfrac{1}{2} \bigl( q^{t-1} + (-1)^h q^\frac{t-2}{2} \rho((-1)^\frac{t}{2}) \bigr)$ & $(1, 2)$ \\
          &                                      $q^{t-1}$                                       & $(2, 1)$ \\
\hline
\end{tabular}
\end{table}

Equation \eqref{eq:PmsumNt} now gives $P_m(a, 2, h)$ explicitly when the structure of the factorization of $m$ is known. 
In particular, if $m>2$ is prime then $P_m(a, 2, h) = \tfrac{1}{m} (N_m - N_1)$, and $d = l = 1$ for both $t=1$, $m$. 
First, from Table \ref{tbl:Nts2}, $N_m = \tfrac{1}{2} (q^{m-1} - 1)$ for $a=0$ and 
\[
  N_m = \tfrac{1}{2} \bigl( q^{m-1} + (-1)^h q^\frac{m-1}{2} \rho((-1)^\frac{m-1}{2} a) \bigr)
\]
for $a \ne 0$. 
If $m=p$ then by Lemma \ref{le:Ntspecial} $N_1 = \tfrac{q-1}{2}$ for $a=0$ and $N_1 = 0$ for $a \ne 0$. 
If $m \ne p$ then Table \ref{tbl:Nts2} yields $N_1 = 0$ for $a=0$ and $N_1 = \tfrac{1}{2} \bigl( 1 + (-1)^h \rho(ma) \bigr)$ for $a \ne 0$. 
Combining these we obtain
\[
  P_m(0, 2, h) =
  \begin{cases}
    \tfrac{1}{2p} (q^{p-1} - q) &\text{if $m=p$,}   \\
    \tfrac{1}{2m} (q^{m-1} - 1) &\text{if $m \ne p$,}
  \end{cases}
\]
and, for $a \ne 0$,
\[
  P_m(a, 2, h) =
  \begin{cases}
    \tfrac{1}{2p} (q^{p-1} + S) &\text{if $m=p$.} \\
    \tfrac{1}{2m} (q^{m-1} + S - (-1)^h \rho(ma) - 1) &\text{if $m \ne p$,}
  \end{cases}
\]
where $S= (-1)^h q^\frac{m-1}{2} \rho((-1)^\frac{m-1}{2} a)$. 
These results are in accordance with \cite[eqs.\ (5.8) and (5.9)]{Carl}.


\subsection{Case $s = 4 = 2^2$}

For $s=4$ we assume that $q=p$, i.e.\ $r=1$. 
Then \cite[Theorem 10.2.5]{BEW} applies directly. 
The more general $q$ will be considered in a future work. 
Since $4 \mid (q-1)$, $p = 4f+1$ for some $f \in \Z$. 
This time we have six possibilities for $d$ and $l$:
\[
  (d,l) = 
  \begin{cases}
    (1,1) &\text{if $\tfrac{m}{t}$ and $t$ are odd,}                       \\
    (1,2) &\text{if $\tfrac{m}{t}$ odd and $t \equiv 2 \; (\bmod \, 4)$,}  \\
    (1,4) &\text{if $\tfrac{m}{t}$ odd and $4 \mid t$,}                    \\
    (2,1) &\text{if $\tfrac{m}{t} \equiv 2 \; (\bmod \, 4)$ and $t$ odd,}  \\
    (2,2) &\text{if $\tfrac{m}{t} \equiv 2 \; (\bmod \, 4)$ and $t$ even,} \\
    (4,1) &\text{if $4 \mid \tfrac{m}{t}$.}                                \\
  \end{cases}
\]
Let $\cn$ be the multiplicative character of order $4$ of $\Fq = \F_{\!p}$ satisfying $\cn(g) = \imu$. 
Furthermore, let $a_4$ and $b_4$ be integers satisfying (see Theorems 3.2.1 and 3.2.2 in \cite{BEW})
\[
  a_4^2 + b_4^2 = p \tc \quad a_4 \equiv -\tLeg{2}{p} \; (\bmod \, 4) \tc \quad b_4 \equiv a_4 g^\frac{p-1}{4} \; (\bmod \, p) \tc
\]
where $\tLeg{2}{p}$ denotes the Legendre symbol. 
Set
\begin{align}\label{eq:pi4def}
  \pn = (-1)^f (a_4 + \imu b_4) \in \Z[\imu]\tp
\end{align}
Then $\pn \cpn = p$ and, since $q=p$,
\begin{align}\label{eq:Jtchi4}
  J_t(\cn) = 
  \begin{cases}
    -p^{\frac{t-4}{4}} \pn^{\frac{t}{2}}         &\text{if $t\equiv 0 \pmod 4$,} \\
    p^{\frac{t-1}{4}} \pn^{\frac{t-1}{2}}        &\text{if $t\equiv 1 \pmod 4$,} \\
    p^{\frac{t-2}{4}} \pn^{\frac{t}{2}}          &\text{if $t\equiv 2 \pmod 4$,} \\
    (-1)^f p^{\frac{t-3}{4}} \pn^{\frac{t+1}{2}} &\text{if $t\equiv 3 \pmod 4$}
  \end{cases}
\end{align}
by \cite[Therem 10.2.5]{BEW}. 
Note also that $\cn^2 = \rho$ (see $s=2$), $\cn^3 = \ccn$, and consequently $J_t(\cn^3) = \conjl{J_t(\cn)}$. 
Further, $\rho(-1) = \cn^2 (-1) = 1$ and $q=p$, so \eqref{eq:Jtrho2} simplifies into
\begin{align} \label{eq:Jtrho4}
  J_t (\cn^2) = J_t (\rho) =
  \begin{cases}
    -p^\frac{t-2}{2} &\text{if $t$ is even,} \\
    p^\frac{t-1}{2}  &\text{if $t$ is odd.}  \\
  \end{cases}
\end{align}

Let us now assume \eqref{eq:restpd} and compute the numbers $M_t$ and $N_t$. 
As in the previous subsection, $N_t$ is obtained in the other cases from Lemma \ref{le:Ntspecial}.
We use the above results on Jacobi sums, and Lemmas \ref{le:MtJacobi0} and \ref{le:MtJacobine0} in the cases $a=0$ and $a \ne 0$, respectively. 
Let first $a=0$. 
If $l=1$, Lemma \ref{le:MtJacobi0} gives $M_t = 1-p$. 
In the case $l=2$ we have 
\[
  M_t = (p-1) \bigl( -1 +(-1)^t p J_t(\rho) \conj{\rho}(g^{i_0}) \bigr)
\] 
by Lemma \ref{le:MtJacobi0}. 
Here $\conj{\rho}(g^{i_0}) = \rho(g^{i_0}) = (-1)^{i_0}$. 
As now $t$ is even, \eqref{eq:Jtrho4} gives $M_t$. 
Finally, if $l=4$, 
\begin{align*}
 \frac{M_t}{p-1} &= -1 + (-1)^t p \sum_{\lambda\in H_4^\ast} J_t(\lambda)\conj{\lambda}(g^{i_0})                                        \\
                 &= -1 + (-1)^t p \bigl( J_t( \cn ) \cn^3( g^{i_0}) + J_t(\rho) \conj{\rho}(g^{i_0}) + J_t( \cn^3 ) \cn( g^{i_0} ) \bigr) \\
                 &= -1 + (-1)^t p \bigl( (-1)^{i_0} J_t( \rho ) + 2\Real( J_t( \cn ) \imu^{ 3 i_0 } ) \bigr)\tp
\end{align*}
The formula for $M_t$ is obtained from this by using \eqref{eq:Jtchi4} and \eqref{eq:Jtrho4} and by remembering that 
$4 \mid t$ in the case $l=4$. 

Let us next consider the case $a \ne 0$. 
If $\tfrac{s}{d}=1$, $M_t = 1$ by Lemma \ref{le:MtJacobine0}. 
This corresponds to $(d,l)=(4,1)$. 
If $\tfrac{s}{d} = 2$ then $(d,l)=(2,1)$ and $t$ is odd or $(d,l)=(2,2)$ and $t$ is even. 
Again by Lemma \ref{le:MtJacobine0} 
\[
  M_t = 1 +(-1)^{t-1} p J_t( \rho ) \rho( (-a_0)^t g^{i_0} )\tp
\]
Here $\rho( (-a_0)^t g^{i_0} )=(-1)^{i_0}\rho( (-1)^t) \rho(a_0^t) = (-1)^{i_0}\rho(a_0)$ when $t$ is odd 
and $\rho( (-a_0)^t g^{i_0} )=(-1)^{i_0}$ when $t$ is even. 
The equation \eqref{eq:Jtrho4} now gives $M_t$. 
If $\tfrac{s}{d}=4$ we have three possibilities for $(d,l)$. 
In the cases $(d,l) = (1,2)$, $(1,4)$ we know $t$ modulo $4$ but in the case $(d,l)=(1,1)$ there are two possibilities: 
$t \equiv 1 \pmod 4$ or $t \equiv 3 \pmod 4$. 
Lemma  \ref{le:MtJacobine0} now gives 
\[
  M_t = 1 + (-1)^{t-1} p \bigl( 2 \Real( J_t( \cn ) {\cn}( (-a_0)^{3t} ) \imu^{3 i_0} ) + (-1)^{i_0} J_t(\rho) \rho(a_0^t) \bigr)\tp 
\]
Again numbers $M_t$ can be obtained from this by using the knowledge on $t$ modulo $4$ and the equations \eqref{eq:Jtchi4} and \eqref{eq:Jtrho4}.

We summarize these results in the following theorem.

\begin{theo} \label{th:Nts=4}
Assume $q=p$ and (\ref{eq:restpd}), and let
\[
  Q_{t,4} = 
  \begin{cases}
    p^\frac{t-1}{4} \pn^\frac{t-1}{2} \ccn(-a_0) \imu^{i_0}        &\text{if $t \equiv 1 \pmod 4$,} \\
    (-1)^f p^\frac{t-3}{4} \pn^\frac{t+1}{2} \cn(-a_0) \imu^{i_0}  &\text{if $t \equiv 3 \pmod 4$.}
  \end{cases}
\]
In addition, let $\pn$ be as in (\ref{eq:pi4def}), $i_0$ as in (\ref{eq:i0def}) and let $a_0 = -\tfrac{m}{ta}$. 
Then the values of $N_t$ for $s=4$ are those listed in Table \ref{tbl:Nts4}.
\end{theo}

\begin{table}[!htb]
\centering
\caption{Values of $N_t$ for $s=4$ assuming \eqref{eq:restpd} and $q=p$.} 
\label{tbl:Nts4}
\begin{tabular}{ccc}
\hline
 $a$  & $N_t$                                                                             & $(d,l)$ \\
\hline
$a=0$ & $\tfrac{1}{4}(p^{t-1}-1)$                                                         & $(1,1)$ \\ 
      & $\tfrac{1}{4}\bigl( p^{t-1} - 1 - (-1)^{i_0} p^{\frac{t-2}{2}} (p-1) \bigr)$      & $(1,2)$ \\
      & $\tfrac{1}{4}\bigl( p^{t-1} - 1 - (-1)^{i_0} p^{\frac{t-4}{4}} (p-1) 
          \bigl( p^{\frac{t}{4}} + 2 \Real(\pn^{\frac{t}{2}} \imu^{i_0}) \bigr) \bigr)$   & $(1,4)$ \\
      & $\tfrac{1}{2} (p^{t-1} - 1)$                                                      & $(2,1)$ \\
      & $\tfrac{1}{2} \bigl( p^{t-1} -1 - (-1)^{i_0} p^{\frac{t-2}{2}} (p-1) \bigr)$      & $(2,2)$ \\
      & $ p^{t-1} -1 $                                                                    & $(4,1)$ \\
\hline
$a \ne 0$ & $\tfrac{1}{4} \bigl( p^{t-1} + (-1)^{i_0} (p^{\frac{t-1}{2}}\rho(a_0) 
              + 2 \Real Q_{t,4}) \bigr)$                                                      & $(1,1)$ \\
          & $\tfrac{1}{4}\bigl( p^{t-1} + (-1)^{i_0} p^{\frac{t-2}{4}} \bigl( p^{\frac{t-2}{4}} 
              - 2 \rho(a_0)\Real(\pn^{\frac{t}{2}} \imu^{i_0}) \bigr)\bigr)$                  & $(1,2)$ \\
          & $\tfrac{1}{4}\bigl( p^{t-1} + (-1)^{i_0} p^{\frac{t-4}{4}} 
               \bigl( p^\frac{t}{4} + 2 \Real(\pn^{\frac{t}{2}} \imu^{i_0}) \bigr) \bigr)$    & $(1,4)$ \\
          & $\tfrac{1}{2} \bigl( p^{t-1} +(-1)^{i_0} p^{\frac{t-1}{2}} \rho(a_0) \bigr)$      & $(2,1)$ \\
          & $\tfrac{1}{2} \bigl( p^{t-1} +(-1)^{i_0} p^{\frac{t-2}{2}} \bigr)$                & $(2,2)$ \\
          & $ p^{t-1} $                                                                       & $(4,1)$ \\
\hline
\end{tabular}
\end{table}

If $m>2$ is prime then we can use Theorem \ref{th:Nts=4} to obtain $P_m(a, 4, h)$. 
As with $s=2$, $P_m(a, 4, h) = \tfrac{1}{m}(N_m - N_1)$ and it is enough to consider $N_t$ for $t=1$, $m$. 

If $t=1$, we have $d=1$ by the assumption $m>2$, and $l=1$. 
In the case $p = m$ we have by Lemma \ref{le:Ntspecial}
\[
  N_1 = 
  \begin{cases}
    \tfrac{1}{4}(p-1) &\text{if $a=0$,}     \\
    0                 &\text{if $a \ne 0$.}
  \end{cases}
\]
If $p \ne m$, \eqref{eq:restpd} holds and from Table \ref{tbl:Nts4} $N_1 = 0$ for $a=0$ and
\begin{equation*}
   N_1 = \tfrac{1}{4} \bigl( 1 + (-1)^{i_0} \bigl( \rho( a_0 ) + 2 \Real(\ccn(-a_0) \imu^{i_0}) \bigr) \bigr)
\end{equation*}
for $a \ne 0$. 
Here $a_0 = -\frac{m}{ta} = -ma^{-1}$ and, modulo $4$,
\begin{equation*}
  i_0 \equiv 
  \begin{cases}
    h  &\text{if $m \equiv 1 \pmod 4$,} \\
    3h &\text{if $m \equiv 3 \pmod 4$.}
  \end{cases}
\end{equation*}
Thus $\rho(a_0) = \rho(-1) \rho(m) \conj{\rho}(a) = \rho(ma)$ and $\ccn(-a_0) = \ccn(ma^{-1}) = \ccn(ma^3) = \cn(m^3 a)$. 

If $t=m$, 
we have $d=1$ and, by the assumption $m>2$, $l=1$. 
Clearly, \eqref{eq:restpd} holds in this case. 
So for $a=0$ we have $N_m = \tfrac{1}{4}(p^{m-1}-1)$. 
For $ a \ne 0$ we have $i_0\equiv h \pmod 4$ by \eqref{eq:i0def} and $a_0 = -\frac{m}{ta}= -a^{-1}$. 
Thus $\rho(a_0) = \rho(-1) \conj{\rho}(a) = \rho(a)$ and $\cn(-a_0) = \cn(a^{-1}) = \ccn(a) = \cn(a^3)$. 
Table \ref{tbl:Nts4} now yields $N_m$ for both $m\equiv 1 \pmod 4$ and $m\equiv 3 \pmod 4$. 
Note that if $m=p$ then $m = q \equiv 1 \pmod 4$.

Combining the above we have
\[
  P_m(0, 4, h) = 
  \begin{cases}
    \tfrac{1}{4}(p^{p-2} - 1)  &\text{if $m=p$,}     \\
    \tfrac{1}{4m}(p^{m-1} - 1) &\text{ if $m \ne p$}
  \end{cases}
\]
for $a=0$. 
If $a \ne 0$ then
\[
  P_p(a, 4, h) = \tfrac{1}{4p} \bigl( p^{p-1} + (-1)^h \bigl( p^\frac{p-1}{2} \rho(a) 
    + 2 p^\frac{p-1}{4} \Real(\pn^{\frac{p-1}{2}} \cn(a) \imu^h) \bigr) \bigr)
\]
for $m=p$ and
\[
  P_m(a, 4, h) = \tfrac{1}{4m} \bigl( p^{m-1} - 1 + (-1)^h \bigl( \rho(a) \bigl( p^{\frac{m-1}{2}} - \rho(m) \bigr) + 2 \Real R_m \bigr) \bigr)
\]
for $m \ne p$, where 
\[
  R_m = 
  \begin{cases}
    p^\frac{m-1}{4} \pn^\frac{m-1}{2} \cn(a) \imu^h - \cn(m^3 a) \imu^h             &\text{ if $m \equiv 1 \pmod 4$,} \\ 
    (-1)^f p^\frac{m-3}{4} \pn^\frac{m+1}{2} \ccn(a) \imu^h - \cn(m^3 a) \imu^{3h}  &\text{ if $m \equiv 3 \pmod 4$.}
  \end{cases}
\]


\subsection{Case $s=3$}

For $s=3$ we again assume that $q=p$, i.e.\ $r=1$. 
Since $3 \mid (p-1)$, $p \equiv 1 \pmod 3$. 
As for $s=2$, we have three possibilities for $d$ and $l$: 
\[
(d,l) = 
  \begin{cases}
    (1,1) &\text{if $3\nmid \tfrac{m}{t}$, $3 \nmid t$,} \\
    (1,3) &\text{if $3\nmid \tfrac{m}{t}$, $3 \mid t$,}  \\
    (3,1) &\text{if $3\mid \tfrac{m}{t}$.}
  \end{cases}
\]
Let $\ck $ be the multiplicative character of order $3$ of $\Fq$ satisfying $\ck(g) = \zeta := \Nap^{2\pi\imu/3}$. 
Obviously $\cck = \ck^{-1} = \ck^2$ and consequently $\conjl{J_t(\ck)} = J_t(\ck^2)$. 
We also note the useful properties $\ck^2(-1)=\ck((-1)^2)=1$ and $\ck(-1) = \conjl{\ck^2(-1)}=1$. 
Let $a_3$ and $b_3$ be integers satisfying (see Theorems 3.1.1 and 3.1.2 in \cite{BEW})
\[
  a_3^2 + 3 b_3^2 = p \tc \quad a_3 \equiv -1 \; (\bmod \, 3) \tc \quad 3 b_3 \equiv (2g^\frac{p-1}{3} + 1) a_3 \; (\bmod \, p) \tc
\]
and denote 
\begin{align}\label{eq:pi3def}
  \pk = \ck (2)(a_3 + \imu b_3 \sqrt{3} ) \in \Z[\zeta] \tp
\end{align}
Since we assume $q=p$, \cite[Theorem 10.2.4]{BEW} is applicable and it yields together with $\pk \cpk = p$ that
\begin{align}\label{eq:Jacs3}
J_t(\ck ) =
  \begin{cases}
    -p^{\frac{t-3}{3}}\pk^{\frac{t}{3}} &\text{if $t \equiv 0 \pmod 3$,} \\
    p^{\frac{t-1}{3}}\pk^{\frac{t-1}{3}} &\text{if $t \equiv 1 \pmod 3$,} \\
    p^{\frac{t-2}{3}}\pk^{\frac{t+1}{3}} &\text{if $t \equiv 2 \pmod 3$.}
  \end{cases}
\end{align}
Let us now assume \eqref{eq:restpd} and compute the numbers $M_t$ and $N_t$. 
Again, in the other cases $N_t$ is obtained from Lemma \ref{le:Ntspecial}.

If $(d,l)=(1,1)$ then $t\equiv 1 \pmod 3$ or $t\equiv 2 \pmod 3$. In the case $a=0$ we again obtain $M_t=1-q$ by Lemma \ref{le:Mtbasrep}. 
For $a\ne0$ Lemma \ref{le:MtJacobine0} gives 
\begin{align*}
  \tfrac{(-1)^{t-1}}{q} (M_t-1) &= \sum_{\lambda \in H_3^\ast} J_t(\lambda) \conjl{\lambda}((-a_0)^t g^{i_0})  \\
                                &= J_t(\ck) \ck^2((-a_0)^t g^{i_0}) + \conjl{J_t(\ck) \ck^2((-a_0)^t g^{i_0})} \\
                                &= 2 \Real(J_t(\ck) \ck(a_0^{2t}) \zeta^{2i_0}) \tp
\end{align*}
As \eqref{eq:Jacs3} tells the value of the Jacobi sum in the above equation, $M_t$ and $N_t$ are easily obtained from this.

If $(d,l) = (1,3)$ then $t \equiv 0 \pmod 3$. 
The numbers $N_t$ can again be obtained as above using Jacobi sums and Lemmas \ref{le:MtJacobi0} and \ref{le:MtJacobine0} 
for $a=0$ and $a \ne 0$, respectively. 
Finally, for $(d,l) = (3,1)$ Lemmas \ref{le:MtJacobi0} and \ref{le:MtJacobine0} give $M_t = 1-q$ if $a=0$, 
and $M_t = 1$ if $a \ne 0$.

Again, Lemma \ref{le:NtsumMt} completes the following theorem. 

\begin{theo}
Assume $q=p$ and (\ref{eq:restpd}), and let
\[
  Q_{t,3} = 
  \begin{cases}
    p^\frac{t-1}{3} \pk^\frac{t-1}{3} \cck(a_0) \zeta^{2i_0} \tc &\text{if $t\equiv 1 \pmod 3$,} \\
    p^\frac{t-2}{3} \pk^\frac{t+1}{3} \ck(a_0) \zeta^{2i_0} \tc  &\text{if $t\equiv 2 \pmod 3$.}
  \end{cases}
\]
Further, let $\pk$ be as in (\ref{eq:pi3def}), $i_0$ as in (\ref{eq:i0def}), and let $a_0 = -\tfrac{m}{ta}$. 
Then the values of $N_t$ for $s=3$ are those listed in Table \ref{tbl:Nts3}.
\end{theo}

\begin{table}[!htb]
\centering
\caption{Values of $N_t$ for $s=3$ assuming \eqref{eq:restpd} and $q=p$.} 
\label{tbl:Nts3}
\begin{tabular}{ccc}
\hline
 $a$  &                                                  $N_t$                                                  & $(d,l)$ \\
\hline
$a=0$ &                                        $\tfrac{1}{3}(p^{t-1}-1)$                                        & $(1,1)$ \\
      & $\tfrac{1}{3} \bigl( p^{t-1} -1 - 2(-1)^t p^{\frac{t-3}{3}} (p-1) \Real(\pk^{\frac{t}{3}}\zeta^{2 i_0}) \bigr)$ 
                                                                                                                & $(1,3)$ \\
      &                                              $p^{t-1} - 1$                                              & $(3,1)$ \\
\hline
$a \ne 0$ &                           $\tfrac{1}{3} \bigl( p^{t-1} - 2(-1)^t \Real Q_{t,3} \bigr)$                  & $(1,1)$ \\
          & $\tfrac{1}{3} \bigl( p^{t-1} + 2(-1)^t p^{\frac{t-3}{3}} \Real(\pk^{\frac{t}{3}} \zeta^{2 i_0}) \bigr)$ 
            & $(1,3)$ \\
          &                                                $p^{t-1}$                                                & $(3,1)$ \\
\hline
\end{tabular}
\end{table}

Again, we shall finally consider the situation when $m>3$ is prime. 
The computations are straightforward and similar as in the case $s=4$ so we just state the results:
\[
  P_m(0, 3, h) = 
  \begin{cases}
    \tfrac{1}{3} (p^{p-2} - 1) &\text{if $m=p$,}     \\
    \tfrac{1}{3m} (p^{m-1} - 1) &\text{if $m \ne p$,}
  \end{cases}
\]
for $a=0$ and
\[
  P_m(a, 3, h) =
  \begin{cases}
    \tfrac{1}{3p} \bigl( p^{p-1} + 2p^\frac{p-1}{3} \Real(\pk^\frac{p-1}{3} \ck(a) \zeta^{2h}) \bigr) &\text{if $m=p$,} \\
    \tfrac{1}{3m} (p^{m-1} - 1 + 2\Real L_m) &\text{if $m \ne p$,}
  \end{cases}
\]
for $a \ne 0$, where
\[
  L_m = 
  \begin{cases}
    \bigl( (p \pk)^\frac{m-1}{3} - \cck(m) \bigr) \ck(a) \zeta^{2h}
      &\text{if $m \equiv 1 \pmod 3$,} \\
    \bigl( p^\frac{m-2}{3} \pk^\frac{m+1}{3} \ck(a) \zeta^h - \cck(m) \bigr) \ck(a) \zeta^h 
      &\text{if $m \equiv 2 \pmod 3$.}
  \end{cases}
\]


\subsection{Case $s \mid (p^e + 1)$}

Assume $r = 2en$ and let $s>1$ be a factor of $p^e + 1$. 
Then $-1$ is a power of $p$ in $\Z_s$, and $s$ is called \emph{semiprimitive}. 
The semiprimitive numbers $N$ appear also in \cite{BauMcE, VluHie} in connection to semiprimitive cyclic codes. 
We recall Theorem 1 in \cite{Moinote}, which we shall use in the following form:

\begin{prop} \label{pr:genspr}
If $s \mid (p^e + 1)$ and $r = 2en$ then
\[
  \sum_{x \in \Fqa[t]} e_t(\gamma_t^i x^s) = 
  \begin{cases}
    (-1)^{nt} \sqrt{q^t} - 1         &\text{if $i \not\equiv k_s \pmod s$,} \\
    (-1)^{nt-1} (s-1) \sqrt{q^t} - 1 &\text{if $i \equiv k_s \pmod s$,}
  \end{cases}
\]
where $k_s = s/2$ if $p>2$, $2 \nmid nt$ and $2 \nmid (p^e + 1)/s$, and $k_s = 0$ otherwise.
\end{prop}

Note that Proposition \ref{pr:genspr} holds for $s=1$, too.

If $a=0$, Lemma \ref{le:Mtbasrep} and Proposition \ref{pr:genspr} immediately give
\begin{align}
  \frac{M_t}{q-1} + 1 = 
  \begin{cases}
    (-1)^{nt} \sqrt{q^t}         &\text{if \eqref{eq:l>i0not} holds,} \\
    (-1)^{nt-1} (l-1) \sqrt{q^t} &\text{if \eqref{eq:l=1i0} holds,}
  \end{cases} \label{eq:Mtspra0}
\end{align}
where the conditions are 
\begin{gather}
  \text{$l>1$ and $i_0 \not\equiv k_l \; (\bmod \, l)$,} \label{eq:l>i0not} \\
  \text{$l=1$; \; or \; $l>1$ and $i_0 \equiv k_l \; (\bmod \, l)$.} \label{eq:l=1i0}
\end{gather}

Assume next that $a \ne 0$. 
We combine Proposition \ref{pr:genspr} with Lemma \ref{le:Mtmonsum} and observe first 
that the congruence $\ind_{\gamma_t} a_0 + t_0 j + i_0 \equiv k_{s/d} \pmod{\tfrac{s}{d}}$ is solvable in $j$ if and only if
\begin{align}\label{eq:jcong}
  l \mid (k_{s/d} - i_0 - \ind_{\gamma_t} a_0) \quad \text{and} \quad 
    \tfrac{t_0}{l} j \equiv \tfrac{k_{s/d} - i_0 - \ind_{\gamma_t} a_0}{l} \pmod u \tp
\end{align}  

Assume first that $l \nmid (k_{s/d} - i_0 - \ind_{\gamma_t} a_0)$. 
Now, by Lemma \ref{le:Mtmonsum} and Proposition \ref{pr:genspr}, we get
\begin{align}
  uM_t &= \bigl( (-1)^{nt} \sqrt{q^t} - 1 \bigr) \sum_{c \in \Fqa} \sum_{j=0}^{u-1} e_1(g^j c^u) \nonumber\\
       &= \bigl( (-1)^{nt-1} \sqrt{q^t} + 1 \bigr) u \tp \label{eq:Mtsprlndiv}
\end{align}
    
Assume next that $l \mid (k_{s/d} - i_0 - \ind_{\gamma_t} a_0)$. 
Since the congruence in \eqref{eq:jcong} has unique solution $j_0 \in \{ 0, \dots, u-1 \}$, Lemma \ref{le:Mtmonsum} 
and Proposition \ref{pr:genspr} imply
\begin{align*}
  uM_t &= \bigl( (-1)^{nt-1} (\tfrac{s}{d} - 1) \sqrt{q^t} - 1 \bigr) \sum_{c \in \Fqa} e_1(g^{j_0} c^u) \\
  &\quad+ \bigl( (-1)^{nt} \sqrt{q^t} - 1 \bigr) \sum_{j \ne j_0} \sum_{c \in \Fqa} e_1(g^j c^u) \tp
\end{align*}
Here
\[
  \sum_{j \ne j_0} \sum_{c \in \Fqa} e_1(g^j c^u) = \sum_{j=0}^{u-1} \sum_{c \in \Fqa} e_1(g^j c^u)
    - \sum_{c \in \Fqa} e_1(g^{j_0} c^u) \tc
\]
and therefore
\[
  uM_t = (-1)^{nt-1} \frac{s}{d} \sqrt{q^t} \sum_{c \in \Fqa} e_1(g^{j_0} c^u) + \bigl( (-1)^{nt-1} \sqrt{q^t} + 1 \bigr) u \tp
\]
Finally, by applying Proposition \ref{pr:genspr} with $t=1$, we get
\[
  \sum_{c \in \Fqa} e_1(g^{j_0} c^u) = 
  \begin{cases}
    (-1)^{n} \sqrt{q} - 1         &\text{if \eqref{eq:sdj0not} holds,} \\
    (-1)^{n-1} (u-1) \sqrt{q} - 1 &\text{if \eqref{eq:sdorj0} holds.}
  \end{cases}
\]
where the conditions are 
\begin{gather}
  \text{$u>1$ and $j_0 \not\equiv k_u \; (\bmod \, u)$,} \label{eq:sdj0not} \\
  \text{$u=1$; \; or \; $u>1$ and $j_0 \equiv k_u \; (\bmod \, u)$.} \label{eq:sdorj0}
\end{gather}

Altogether, if $l \mid (k_{s/d} - i_0 - \ind_{\gamma_t} a_0)$, then
\begin{align}
  &M_t - 1 = \label{eq:Mtsprldiv} \\
  &\begin{cases}
    (-1)^{nt-1} \bigl( ((-1)^n \sqrt{q} - 1)l + 1 \bigr) \sqrt{q^t}           &\text{if \eqref{eq:sdj0not} holds,} \\
    (-1)^{nt-1} \bigl( ((-1)^{n-1} (u-1) \sqrt{q} - 1)l + 1 \bigr) \sqrt{q^t} &\text{if \eqref{eq:sdorj0} holds.}
  \end{cases} \nonumber
\end{align}

Combining \eqref{eq:Mtspra0}, \eqref{eq:Mtsprlndiv} and \eqref{eq:Mtsprldiv} with Lemma \ref{le:NtsumMt} we get the values of $N_t$ 
which we gather in the following theorem.

\begin{theo}
Assume $s \mid (p^e + 1)$ and $r = 2en$. 
Then the $N_t$ are those listed in Table \ref{tbl:Ntssemipr}. 
Especially, if $a=0$ and $l=1$ then $N_t = \tfrac{d}{s} (q^{t-1} -1)$, and if $a \ne 0$ and $d=s$ then $N_t = q^{t-1}$.
\end{theo}

\begin{table}[!htb]
\centering
\caption{Values of $N_t$ for $s \mid (p^e + 1)$ and $r = 2en$ with ``$\mid$'' and ``$\nmid$'' telling 
         whether $l$ divides $k_{s/d} - i_0 - \ind_{\gamma_t} a_0$ or not.} 
\label{tbl:Ntssemipr}
\begin{tabular}{c@{}cc}
\hline
 $a$  & $N_t$ & with \\
\hline
$a=0$ & $\tfrac{d}{s} \bigl( q^{t-1} - 1 + (-1)^{nt} (q-1) \sqrt{q^{t-2}} \bigr)$ & \eqref{eq:l>i0not} \\
      & \rule[-1ex]{0pt}{2ex}$\tfrac{d}{s} \bigl( q^{t-1} - 1 - (-1)^{nt} (q-1)(l-1) \sqrt{q^{t-2}} \bigr)$ & \eqref{eq:l=1i0} \\
\hline
$a \ne 0$ & $\tfrac{d}{s} \bigl( q^{t-1} - (-1)^{nt} \sqrt{q^{t-2}} \bigr)$ & $\nmid$ \\
          & $\tfrac{d}{s} \bigl( q^{t-1} - (-1)^{nt} \bigl( ((-1)^n \sqrt{q} - 1)l + 1 \bigr) \sqrt{q^{t-2}} \bigr)$ 
          & $\mid$, \eqref{eq:sdj0not} \\
          & \rule[-1ex]{0pt}{2ex}$\tfrac{d}{s} \bigl( q^{t-1} - (-1)^{nt} \bigl( ((-1)^{n-1} (u-1) \sqrt{q} - 1)l + 1 \bigr) \sqrt{q^{t-2}} \bigr)$ 
          & $\mid$, \eqref{eq:sdorj0} \\
\hline
\end{tabular}
\end{table}


\subsection{The semiprimitive and index $2$ cases for $p=2$} \label{sbs:sprind2}

In this subsection we assume that $p=2$ and show how to calculate $P_m (a, q-1, h)$ in semi\-prim\-itive or index $2$ cases by applying 
the results from \cite[II]{MoiThe} and \cite{RinThe}. 
In particular, we give $P_m (0, q-1, h)$ explicitly for all $m \le 30$. 
We also give a table of these numbers for $q=2$, $4$, $8$, and small values of $m$ to cross-check our formulae against the results given by 
the irreducible polynomial generator in \cite{polgen}.

As $p=2$, the semiprimitive case holds for an odd integer $N>1$ if $-1$ is a power of $2$ in $\Z_N$. 
Correspondingly, the \emph{index $2$ case} is said to hold for $N$ if $-1 \notin \gen{2} \subseteq \Z_N$ and $\ord_N 2 = \phi(N)/2$ 
where $\phi$ is the Euler function.

If $s$ is semiprimitive then clearly its factors, especially $l$, $s/d$ and $u$ in Lemmas \ref{le:Mtbasrep} and \ref{le:Mtmonsum}, are too. 
Proposition \ref{pr:genspr} can be written for characteristic $p=2$ in the following form, see also \cite[II Theorem 1]{MoiThe}. 

\begin{prop} \label{pr:sumspr}
Assume that $rt = N' \ord_N 2$, $N>1$ and $-1$ is a power of $2$ modulo $N$. 
Then
\[
  \sum_{x \in \Fq[t]} e_t(\gamma_t^a x^N) =
  \begin{cases}
    (-1)^{N'} \sqrt{q^t}         &\text{if $N \nmid a$,} \\
    (-1)^{N'-1} (N-1) \sqrt{q^t} &\text{if $N \mid a$.}
  \end{cases}
\]
\end{prop}

Similarly, if the index $2$ case holds for $N$ then its factors satisfy either the index $2$ or the semiprimitive case, 
see \cite[II Lemmas 2 and 5]{MoiThe}. 
In what follows we consider only square-free $N$ in the index $2$ cases. 
From the general classification result \cite[II Lemmas 3 and 6]{MoiThe} it follows that the following three cases are then possible, 
where $p_1$ and $p_2$ are primes:

\begin{enumerate}[\hspace{1.5em} 1.]
\item\label{it:Np}
$N = p_2 \equiv 7 \pmod 8$; 

\item\label{it:Npqroot}
$N = p_1 p_2$, $p_1 \equiv 5 \pmod 8$, $p_2 \equiv 3 \pmod 8$, $2$ is a primitive root modulo $p_1$ and modulo $p_2$;

\item\label{it:Npqnoro}
$N = p_1 p_2$, $p_1 \equiv 3\tc\ 5 \pmod 8$, $p_2 \equiv 7 \pmod 8$, $\ord_{p_1} 2 = p_1 - 1$, and $\ord_{p_2} 2 = (p_2 - 1)/2$ 
with $-1 \notin \gen{2} \subseteq \Z_{p_2}$.
\end{enumerate}

The value distribution of the monomial sums in the above square-free cases were studied in \cite{BauMyk} (case \ref{it:Np}) 
and \cite{VluHD} (cases \ref{it:Npqroot} and \ref{it:Npqnoro}). 
The general (characteristic $2$) index $2$ cases has been studied in \cite{MoiThe} (cases \ref{it:Np} and \ref{it:Npqroot}) 
and in \cite{RinThe} (case \ref{it:Npqnoro}). 
We are able to compute the value distribution except for few case \ref{it:Npqnoro} parameters. 
Knowing only the value distribution and not the exact values is not enough to compute the $P_m(a, s, h)$ exactly but with
the methods from \cite{MoiThe} and \cite{RinThe} we get (except for some cases \ref{it:Npqnoro}) at most two possibilities for the values 
of each $P_m(a, s, h)$ when the index $2$ case holds for $m$. 

Let us next recall how the index $2$ sums $\sum_{x \in \Fq[t]} e_t(\gamma_t^i x^N)$ can be computed in our three cases. 
For the results and methods we refer to \cite{MoiVaa, MoiThe} for cases \ref{it:Np} and \ref{it:Npqroot} and to \cite{RinThe}, 
especially Theorems 4, 6 and 7, for case \ref{it:Npqnoro}. 
Also \cite{BauMyk, VluHD} can be used. 
Let $rt = r' t'$ with $r' = \phi(N)/2 = \ord_N 2$ and denote by $\delta = \Norm (\gamma_t)$ a primitive element of $\F_{\!2^{r'}}$, 
where $\Norm$ is the norm from $\Fq[t]$ onto $\F_{\!2^{r'}}$. 
Further, since $N \mid (2^{r'} - 1)$, there exists a multiplicative character $\chi$ of $\F_{\!2^{r'}}$ 
for which $\chi(\delta) = \Nap^{2\pi \imu/N}$. 
The character $\chi$ has order $N$ and $\chi' = \chi \circ \Norm$ is a multiplicative character of order $N$ of $\Fq[t]$.

The value of the monomial index $2$ sum $\sum_{x \in \Fq[t]} e_t(\gamma_t^i x^N)$ can now be computed in terms 
of $G_t (\chi) = \sum_{x \in \Fqa[t]} e_t(x) \chi'(x)$ by using \cite[Theorem 2]{MoiThe} in the case \ref{it:Np}, 
\cite[Theorem 3]{MoiThe} in the case \ref{it:Npqroot} and \cite[eq.~(16), Theorem 4]{RinThe} in the case \ref{it:Npqnoro}. 
By the Davenport-Hasse identity
\begin{align} \label{eq:Gaussdown}
  G_t (\chi) = -(-F_{r'} (\chi))^{t'}\tc \qquad F_{r'}(\chi) = \sum_{x \in \F_{\!2^{r'}}} \chi(x) e(x) \tc
\end{align}
where in the last Gauss sum over $\F_{\!2^{r'}}$ $e$ is the canonical additive character of $\F_{\!2^{r'}}$. 
These latter Gauss sums can be computed up to the sign of the imaginary part, see \cite[p.~1245]{MoiVaa} and \cite[p.~9 and Theorem 7]{RinThe}.

The above cases cover all values $m \le 30$, so we able to compute (possibly up to two choices) $P_m (0, s, h)$ 
for $m \le 30$ by Lemma \ref{le:Mtbasrep} and $P_m (a, s, h)$ for $a \ne 0$, $m \le 30$, by Lemma \ref{le:Mtmonsum}. 
As an example we give $P_m := P_m(0, q-1, h)$ for $m \le 30$. 
Since $s = q-1$, we have $b$ fixed and $h = \ind b = \ind_g b$. 
We consider the values $m \le 30$ in the following order: $2^k$ ($2$, $4$, $8$, $16$), semiprimitive primes $v$ ($3$, $5$, $11$, $13$, 
$17$, $19$, $29$) and the cases related to these: $2v$ ($6$, $10$, $22$, $26$), $4v$ ($12$, $20$), $8v$ ($24$), $v^2$ ($9$, $25$), 
$2v^2$ ($18$), $v^3$ ($27$). 
Finally, we cover the index $2$ cases: $7$, $23$ with related $14$, $28$ (case \ref{it:Np}), $15$ with related $30$ (case \ref{it:Npqroot}), 
and $21$ (case \ref{it:Npqnoro}).

If $m = 2^k$ then \eqref{eq:PmsumNt} gives $P_m = \tfrac{1}{m} (N_m - N_{m/2})$. 
By Lemma \ref{le:Ntspecial} $N_{m/2} = \tfrac{q^{m/2} - 1}{q-1}$. 
Lemma \ref{le:Mtbasrep} gives $M_m = 1-q$ and then $N_m = \tfrac{q^{m-1} - 1}{q-1}$ by Lemma \ref{le:NtsumMt}. 
Thus, as in \cite[Example 2]{MoiKlo},
\[
  P_m = \frac{q^{m-1} - q^{m/2}}{m(q-1)} \tp
\]

In the case $m = v$ equation \eqref{eq:PmsumNt} implies $P_v = \tfrac{1}{v} (N_v - N_1)$. 
If $v \nmid (q-1)$ then $d = l = 1$ for both values $t = 1$, $v$. 
By Lemma \ref{le:Mtbasrep}, $M_1 = M_v = 1-q$ and therefore Lemma \ref{le:NtsumMt} gives $N_1 = 0$ and $N_v = \tfrac{q^{v-1} - 1}{q-1}$. 
Thus
\[
  P_v = \frac{q^{v-1} - 1}{v(q-1)} \tp
\]
Assume now that $v \mid (q-1)$. 
If $t=1$ then $d=v$ and $l=1$. 
By Lemma \ref{le:Ntspecial} $N_1 = 0$ if $v \nmid h$. 
If $v \mid h$ then $M_1 = 1-q$ by Lemma \ref{le:Mtbasrep} and therefore Lemma \ref{le:NtsumMt} gives $N_1 = 0$ in this case, too. 
If $t=v$ then $d=1$ and $l=v$. 
By Lemma \ref{le:Mtbasrep} and Proposition \ref{pr:sumspr} we now have
\[
  M_v =
  \begin{cases}
    (q-1) (-1 \pm \sqrt{q^v})       &\text{if $v \nmid h$,} \\
    (q-1) (-1 \mp (v-1) \sqrt{q^v}) &\text{if $v \mid h$,}  \\
  \end{cases}
\]
where $\pm = (-1)^{v'}$ with $rv = v' \ord_v 2$. 
Since $v$ is odd, $\pm = (-1)^\frac{r}{\ord_v 2}$. 
By using Lemma \ref{le:NtsumMt} and combining the above results we get
\begin{align} \label{eq:P3}
  P_v - \frac{q^{v-1}-1}{v(q-1)} =
  \begin{cases}
    0                                 &\text{if $\ord_v 2 \nmid r$,}                  \\
    \pm \tfrac{1}{v} \sqrt{q^{v-2}}   &\text{if $\ord_v 2 \mid r$, $v \nmid \ind b$,} \\
    \mp \tfrac{v-1}{v} \sqrt{q^{v-2}} &\text{if $\ord_v 2 \mid r$, $v \mid \ind b$.}
  \end{cases}
\end{align}

To consider the case $m = 2v$ we note the following. 
If $\tfrac{m}{2}$ is odd we have $\mu(m) N_1 + \mu(\tfrac{m}{2}) N_2 = 0$. 
Namely, for $t=1$, $d = \gcd(m, q-1) = \gcd(\tfrac{m}{2}, q-1)$, and Lemma \ref{le:Ntspecial} gives $N_1 = 0$ or $d$ 
if $d \nmid h$ or $d \mid h$, respectively. 
For $t=2$ we have $d = \gcd(\tfrac{m}{2}, q-1)$ again and $l = \gcd(2, \tfrac{q-1}{d}) = 1$. 
If $d \nmid h$ then $N_2 = 0$ by Lemma \ref{le:Ntspecial}. 
If $d \mid h$ then Lemma \ref{le:Mtbasrep} gives $M_2 = 1-q$ and therefore $N_2 = \tfrac{d}{(q-1)q} (q^2 - 1 + 1 -q) = d$. 
This proves the claim $\mu(m) N_1 + \mu(\tfrac{m}{2}) N_2 = 0$ for odd $\tfrac{m}{2}$. 
Note that this claim holds true also if $8 \mid m$ or $m$ is non-square-free. 
The use of \eqref{eq:PmsumNt} now gives $P_{2v} = \tfrac{1}{2v} (N_{2v} - N_v)$ by the above consideration.

For $t=v$, $d = \gcd(2, q-1) = 1$ and $2 \mid \tfrac{m}{t}$. 
By Lemma \ref{le:Ntspecial} $N_v = \tfrac{q^v - 1}{q-1}$. 
For $t = 2v$, $d=1$ again and $l = \gcd(2v, q-1) = \gcd(v, q-1)$. 
If $v \nmid (q-1)$ then $l=1$ and Lemmas \ref{le:NtsumMt} and \ref{le:Mtbasrep} yield
\[
  N_{2v} = \frac{1}{(q-1)q} (q^{2v} - 1 + 1 - q) = \frac{q^{2v-1} - 1}{q-1} \tp
\]
If $v \mid (q-1)$ then $l=v$ and Lemma \ref{le:Mtbasrep} and Proposition \ref{pr:sumspr} give
\[
  M_{2v} =
  \begin{cases}
    (q-1) (-1 + q^v)       &\text{if $v \nmid h$,} \\
    (q-1) (-1 - (v-1) q^v) &\text{if $v \mid h$,}  \\
  \end{cases}
\]
since now $\tfrac{2rv}{\ord_v 2}$ is even. 
The use of Lemma \ref{le:NtsumMt} together with these results gives
\[
  P_{2v} - \frac{q^{2v-1} - q^v}{2v(q-1)} =
  \begin{cases}
    0                       &\text{if $\ord_v 2 \nmid r$,}                  \\
    \frac{1}{2v} q^{v-1}    &\text{if $\ord_v 2 \mid r$, $v \nmid \ind b$,} \\
    -\frac{v-1}{2v} q^{v-1} &\text{if $\ord_v 2 \mid r$, $v \mid \ind b$.}  \\
  \end{cases}
\]
For the remaining cases related to semiprimitive primes $v$ the use of Lemmas \ref{le:Ntspecial}, \ref{le:NtsumMt} 
and \ref{le:Mtbasrep} and Proposition \ref{pr:sumspr} gives the following results. 
The details of the calculations are given in \cite{auxcal}.

If $m = 4v$ ($12$, $20$) then
\[
  P_{4v} - \frac{q^{2v} (q^{2v-1} - 1)}{4v(q-1)} =
  \begin{cases}
    -\frac{q^2}{4v}                            &\text{if $\ord_v 2 \nmid r$,}                  \\
    \frac{q^{2v-1}}{4v}                        &\text{if $\ord_v 2 \mid r$, $v \nmid \ind b$,} \\
    -\frac{v-1}{4v} q^{2v-1} - (\frac{q}{2})^2 &\text{if $\ord_v 2 \mid r$, $v \mid \ind b$.}  \\
  \end{cases}
\]

If $m = 24$ then
\[
  P_{24} - \frac{q^{12} (q^{11} - 1)}{24(q-1)} =
  \begin{cases}
    -\frac{q^4 (q^3 - 1)}{24(q-1)}                         &\text{if $2 \nmid r$,}                   \\
    \frac{q^{11}}{24}                                      &\text{if $2 \mid r$, $3 \nmid \ind b$,}  \\
    -\frac{q^{11}}{12} - \frac{q^4 (q^3 - 1)}{8(q-1)}      &\text{if $2 \mid r$, $3 \mid \ind b$.}   \\
  \end{cases}
\]

If $m = v^2$ ($9$, $25$, which are semiprimitive) then
\[
  P_{v^2} - \frac{q^{v^2-1} - 1}{v^2 (q-1)} = \frac{q^{v-1} - 1}{v^2 (q-1)}
\]
for $\ord_v 2 \nmid r$,
\[
  P_{v^2} - \frac{q^{v^2-1} - 1}{v^2 (q-1)} =
  \begin{cases}
    \pm_1 \frac{1}{v^2} \sqrt{q^{v^2-2}}                               &\text{if $v \nmid \ind b$,} \\
    -\frac{q^{v-1} - 1}{v(q-1)} \mp_1 \frac{v-1}{v^2} \sqrt{q^{v^2-2}} &\text{if $v \mid \ind b$}
  \end{cases}
\]
for $\ord_v 2 \mid r$, $\ord_{v^2} 2 \nmid r$, and
\begin{align*}
  &P_{v^2} - \frac{q^{v^2-1} - 1}{v^2 (q-1)} = \\
  &\begin{cases}
    \pm_2 \frac{1}{v^2} \sqrt{q^{v^2-2}} &\text{if $v \nmid \ind b$,}    \\
    \pm_2 \frac{1}{v^2} \sqrt{q^{v^2-2}} - \frac{1}{v} \bigl( \frac{q^{v-1} - 1}{q-1} \pm_1 \sqrt{q^{v-2}} \bigr) 
      &\text{if $v \mid \ind b$, $v^2 \nmid \ind b$,} \\
    \mp_2 \frac{v^2 - 1}{v^2} \sqrt{q^{v^2-2}} - \frac{1}{v} \bigl( \frac{q^{v-1} - 1}{q-1} \mp_1 (v-1) \sqrt{q^{v-2}} \bigr) \hspace{-6pt}
      &\text{if $v^2 \mid \ind b$}
  \end{cases}
\end{align*}
for $\ord_{v^2} 2 \mid r$, where $\pm_i = (-1)^\frac{r}{\ord_{v^i} 2}$ for $i = 1$, $2$.

If $m = 18$ then
\[
  P_{18} - \frac{q^9 (q^8 - 1)}{18(q-1)} = -\frac{q^3 (q+1)}{18}
\]
for $2 \nmid r$,
\[
  P_{18} - \frac{q^9 (q^8 - 1)}{18(q-1)} = 
  \begin{cases}
    \frac{q^8}{18}                    &\text{if $3 \nmid \ind b$,} \\
    -\frac{q^3}{3} (\frac{q^5}{3} + \frac{q+1}{2}) &\text{if $3 \mid \ind b$}
  \end{cases}
\]
for $2 \mid r$ and $6 \nmid r$, and
\begin{align*}
  P_{18} - \frac{q^9 (q^8 - 1)}{18(q-1)} = 
  \begin{cases}
    \frac{q^8}{18}                    &\text{if $3 \nmid \ind b$,} \\
    \frac{q^8}{18} - \frac{q^2(q^3 - 1)}{6(q-1)} &\text{if $3 \mid \ind b$, $9 \nmid \ind b$,} \\
    -\frac{q^2(8q^6 + 3q(q+1) - 6)}{18} &\text{if $9 \mid \ind b$}
  \end{cases}
\end{align*}
for $6 \mid r$.

If $m = 27$ then
\[
  P_{27} - \frac{q^{26} - 1}{27(q-1)} = -\frac{q^8 - 1}{27(q-1)}
\]
for $2 \nmid r$,
\[
  P_{27} - \frac{q^{26} - 1}{27(q-1)} = 
  \begin{cases}
    \pm \frac{1}{27} \sqrt{q^{25}}                                         &\text{if $3 \nmid \ind b$,} \\
    -\frac{1}{27} \bigl( \frac{3(q^8 - 1)}{q-1} \pm 2 \sqrt{q^{25}} \bigr) &\text{if $3 \mid \ind b$}   \\
  \end{cases}
\]
for $2 \mid r$, $6 \nmid r$,
\begin{align*}
  P_{27} - \frac{q^{26} - 1}{27(q-1)} = 
  \begin{cases}
    \pm \frac{1}{27} \sqrt{q^{25}} &\text{if $3 \nmid \ind b$,} \\
    -\frac{q^8 - 1}{9(q-1)} \pm \frac{1}{27} (\sqrt{q^{25}} - 3\sqrt{q^7}) \hspace{-5pt} 
      &\text{if $3 \mid \ind b$, $9 \nmid \ind b$,} \\
    -\frac{q^8 - 1}{9(q-1)} \mp \frac{2}{27} (4 \sqrt{q^{25}} - 3\sqrt{q^7}) \hspace{-6.5pt} 
      &\text{if $9 \mid \ind b$}
  \end{cases}
\end{align*}
for $6 \mid r$, $18 \nmid r$, and
\begin{align*}
  P_{27} &- \frac{q^{26} - 1}{27(q-1)} \\
         &= 
  \begin{cases}
    \pm \frac{1}{27} \sqrt{q^{25}} &\text{if $3 \nmid \ind b$,} \\
    -\frac{q^8 - 1}{9(q-1)} \pm \frac{1}{27} (\sqrt{q^{25}} - 3\sqrt{q^7}) 
      &\text{if $3 \mid \ind b$, $27 \nmid \ind b$,} \\
    -\frac{q^8 - 1}{9(q-1)} \mp \frac{2}{27} (13 \sqrt{q^{25}} - 12\sqrt{q^7}) 
      &\text{if $27 \mid \ind b$}
  \end{cases}
\end{align*}
for $18 \mid r$, where $\pm = (-1)^\frac{r}{2}$.

In the index $2$ case \ref{it:Np}, $m = p_2$ ($7$, $23$) and we have $P_m = \tfrac{1}{m} (N_m - N_1)$ by \eqref{eq:PmsumNt}. 
In considering $N_1$ we have $d = \gcd(m, q-1)$ and $l=1$. 
Thus Lemma \ref{le:Ntspecial} gives $N_1 = 0$ if $d \nmid h$. 
In the case $d \mid h$ the use of Lemmas \ref{le:NtsumMt} and \ref{le:Mtbasrep} implies $N_1 = 0$, too.

If $t=m$ then $d=1$ and $l = \gcd(m, q-1)$. 
If $m \nmid (q-1)$ then $l=1$ and $M_m = 1-q$ by Lemma \ref{le:Mtbasrep}. 
Thus Lemma \ref{le:NtsumMt} gives $N_m = \tfrac{q^{m-1} - 1}{q-1}$. 
For $m \mid (q-1)$ we have
\[
  N_m = \frac{q^{m-1} - 1}{q-1} + \frac{1}{q} \sum_{x \in \Fq[m]} e_m (\gamma_m^{i_0} x^m)
\]
by Lemmas \ref{le:NtsumMt} and \ref{le:Mtbasrep}.

To determine $N_7$ we note that $r' = \phi(7)/2 = 3 = \ord_7 2$ and $t' = 7r/3$ in \eqref{eq:Gaussdown}. 
Further, as given on \cite[p.~3]{MRRV},
\[
  F_3 (\chi) = -1 + c \sqrt{-7} \tc \qquad c \in \{ 1, -1 \} \tc
\]
in \eqref{eq:Gaussdown} and the use of \cite[Lemma 3]{MRRV} togehter with the above consideration gives
\begin{align*}
  P_7 - \frac{q^6-1}{7(q-1)} = 
  \begin{cases}
    0 &\text{if $3 \nmid r$,} \\
    -\frac{3}{7}(\omega_7^{\frac{7r}{3}} + \bar{\omega}_7^{\frac{7r}{3}}) \sqrt{q^5} 
      &\text{if $3 \mid r$, $7 \mid \ind b$,} \\
    \frac{\sqrt{2}}{7}(\omega_7^{\frac{7r}{3} - 1} + \bar{\omega}_7^{\frac{7r}{3} - 1}) \sqrt{q^5} 
      &\text{if $3 \mid r$, $\ind b \in C_c^7$,} \\
    \frac{\sqrt{2}}{7}(\omega_7^{\frac{7r}{3} + 1} + \bar{\omega}_7^{\frac{7r}{3} + 1}) \sqrt{q^5} 
      &\text{if $3 \mid r$, $\ind b \in C_{-c}^7$,}
  \end{cases}
\end{align*}
where $\omega_7 = (1 + \sqrt{-7})/\sqrt{8}$, bar denotes the complex conjugation and $C_i^N$ denotes the $2$-cyclotomic coset 
modulo $N$ containing $i$.

In the related case $m = 2 \cdot 7 = 14$ we see as above in the case $m = 2v$ that $\mu(14) N_1 + \mu(7) N_2 = 0$ in \eqref{eq:PmsumNt}. 
Thus $P_{14} = \tfrac{1}{14} (N_{14} - N_7)$. 
If $t=7$ then $d = \gcd(2, q-1) = 1$. 
Since $\tfrac{m}{t} = 2$, Lemma \ref{le:Ntspecial} can be applied to get $N_7 = \tfrac{q^7 - 1}{q-1}$. 
In the case $t = 14$, $d=1$ and $l = \gcd(14, q-1) = \gcd(7, q-1)$. 
If $7 \nmid (q-1)$ then $l=1$ and Lemmas \ref{le:NtsumMt} and \ref{le:Mtbasrep} give $N_{14} = \tfrac{q^{13} - 1}{q-1}$. 
By using again Lemmas \ref{le:NtsumMt} and \ref{le:Mtbasrep} we have
\[
  N_{14} = \frac{q^{13} - 1}{q-1} + \frac{1}{q} \sum_{x \in \Fq[m]} e_m (\gamma_m^{i_0} x^7)
\]
if $7 \mid (q-1)$. 
Now $rm = m' \ord_7 2$ or $m' = 14r/3$, and therefore we get as above
\begin{align*}
  P_{14} - \frac{q^7 (q^6-1)}{14(q-1)} = 
  \begin{cases}
    0 &\text{if $3 \nmid r$,} \\
    -\frac{3}{14} (\omega_7^{\frac{14r}{3}} + \bar{\omega}_7^{\frac{14r}{3}}) q^6 
      &\text{if $3 \mid r$, $7 \mid \ind b$,} \\
    \frac{\sqrt{2}}{14} (\omega_7^{\frac{14r}{3} - 1} + \bar{\omega}_7^{\frac{14r}{3} - 1}) q^6 
      &\text{if $3 \mid r$, $\ind b \in C_c^7$,} \\
    \frac{\sqrt{2}}{14} (\omega_7^{\frac{14r}{3} + 1} + \bar{\omega}_7^{\frac{14r}{3} + 1}) q^6 
      &\text{if $3 \mid r$, $\ind b \in C_{-c}^7$}
  \end{cases}
\end{align*}
with the same $c$ and $\omega_7$ as in $P_7$.

For the details of the cases $m = 4 \cdot 7 = 28$ and $m = 23$ we refer to \cite{auxcal} and state here the results:
\begin{align*}
  P_{28} &- \frac{q^{14} (q^{13} - 1)}{28(q-1)} \\
         &= 
  \begin{cases}
    -\tfrac{q^2}{28}     &\text{if $3 \nmid r$,}        \\
    -\frac{3}{28}(\omega_7^\frac{28r}{3} + \conj{\omega}_7^\frac{28r}{3}) q^{13} - (\frac{q}{2})^2 
      &\text{if $3 \mid r$, $7 \mid \ind b$,}      \\
    \frac{\sqrt{2}}{28}(\omega_7^{\frac{28r}{3} - 1} + \conj{\omega}_7^{\frac{28r}{3} - 1}) q^{13} 
      &\text{if $3 \mid r$, $\ind b \in C_c^7$,}   \\
    \frac{\sqrt{2}}{28}(\omega_7^{\frac{28r}{3} + 1} + \conj{\omega}_7^{\frac{28r}{3} + 1}) q^{13} 
      &\text{if $3 \mid r$, $\ind b \in C_{-c}^7$} \\
  \end{cases}
\end{align*}
with the same $c$ and $\omega_7$ as in $P_7$. 
For $m = 23$ the Gauss sum $F_{11} (\chi)$ can be calculated with the method described on \cite[p.~3]{MRRV}. 
We obtain $F_{11} (\chi) = 2^3 (-3 + c \sqrt{-23})$, where $c \in \{ 1, -1 \}$. 
Then
\[
  P_{23} - \frac{q^{22} - 1}{23(q-1)} = 
  \begin{cases}
    0 &\text{if $11 \nmid r$,} \\
    -\frac{11}{23}(\omega_{23}^\frac{23r}{11} + \conj{\omega}_{23}^\frac{23r}{11}) q^\frac{69}{11} &\text{if $11 \mid r$, $23 \mid \ind b$,} \\
    \frac{q^\frac{69}{11}}{23} \Real(\omega_{23}^\frac{23r}{11} (1 + \sqrt{-23}))          &\text{if $11 \mid r$, $\ind b \in C_c^{23}$,}    \\
    \frac{q^\frac{69}{11}}{23} \Real(\omega_{23}^\frac{23r}{11} (1 - \sqrt{-23}))          &\text{if $11 \mid r$, $\ind b \in C_{-c}^{23}$,} \\
  \end{cases}
\]
where $\omega_{23} = 3 - \sqrt{-23}$.

The index $2$ case \ref{it:Npqroot} holds for $m = 15$. 
Now $F_4 (\chi) = 1 + c \sqrt{-15}$ in \eqref{eq:Gaussdown} is given in \cite[Lemma 5]{MRRV}, where $c \in \{ 1, -1 \}$.
We again just state the results for $m = 15$ and the related $m = 30$, and refer to \cite{auxcal} for the details. 
If $m = 15$ then
\[
  P_{15} - \frac{q^{14}-1}{15(q-1)} = -\frac{q^4 + q^2 - 2}{15(q-1)}
\]
for $2 \nmid r$,
\[
  P_{15} - \frac{q^{14}-1}{15(q-1)} = 
  \begin{cases}
    -\frac{1}{15} \bigl( q + 1 + \sqrt{q^{13}} - \sqrt{q} \bigr)             &\text{if $3 \nmid \ind b$,} \\
    -\frac{1}{15} \bigl( \frac{3(q^4 - 1)}{q-1} - 2\sqrt{q^{13}} + (\sqrt{q} + 1)^2 \bigr) &\text{if $3 \mid \ind b$}   \\
  \end{cases}
\]
for $2 \mid r$, $4 \nmid r$, and
\begin{align*}
  &P_{15} - \frac{q^{14}-1}{15(q-1)} = \\
  &\begin{cases}
    -\frac{2}{15} \bigl( 2(\omega_{15}^\frac{15r}{4} + \bar{\omega}_{15}^\frac{15r}{4}) + 1 \pm 2 \bigr) \sqrt{q^{13}} \\
    \hspace{3em}- \frac{1}{5} \bigl( \frac{q^4 - 1}{q-1} \mp 4 \sqrt{q^3} \bigr) - \frac{1}{3} (\sqrt{q} - 1)^2 
      &\text{if $15 \mid \ind b$,} \\
    \frac{1}{15} \bigl( \omega_{15}^\frac{15r}{4} + \bar{\omega}_{15}^\frac{15r}{4} - 2 \pm 1 \bigr) \sqrt{q^{13}} 
      - \frac{1}{5} \bigl( \frac{q^4 - 1}{q-1} \pm \sqrt{q^3} \bigr) 
      &\text{if $\ind b \in C_3^{15}$,} \\
    \frac{1}{15} \bigl( 2(\omega_{15}^\frac{15r}{4} + \bar{\omega}_{15}^\frac{15r}{4}) + 1 \mp 4 \bigr) \sqrt{q^{13}} 
      - \frac{1}{3} (q + 1 + \sqrt{q}) 
      &\text{if $\ind b \in C_5^{15}$,} \\
    \frac{1}{15} \bigl( 2(\omega_{15}^{\frac{15r}{4} + 1} + \bar{\omega}_{15}^{\frac{15r}{4} + 1}) + 1 \pm 1 \bigr) \sqrt{q^{13}} 
      &\text{if $\ind b \in C_c^{15}$,} \\
    \frac{1}{15} \bigl( 2(\omega_{15}^{\frac{15r}{4} - 1} + \bar{\omega}_{15}^{\frac{15r}{4} - 1}) + 1 \pm 1 \bigr) \sqrt{q^{13}} 
      &\text{if $\ind b \in C_{-c}^{15}$}
  \end{cases}
\end{align*}
for $4 \mid r$, where $\pm = (-1)^\frac{r}{4}$ and $\omega_{15} = -(1 + \sqrt{-15})/4$.

If $m = 30$ then
\[
  P_{30} - \frac{q^{15} (q^{14} - 1)}{30(q-1)} = -\frac{q^3 (q^6 - 1)}{30(q-1)}
\]
for $2 \nmid r$,
\[
  P_{30} - \frac{q^{15} (q^{14} - 1)}{30(q-1)} = 
  \begin{cases}
    -\frac{q^3 (q^2 - 1)}{30(q-1)} + \frac{q^2}{30} (q^{12} - 1)           &\text{if $3 \nmid \ind b$,} \\
    -\frac{q^3 (3 q^6 - 2 q^2 - 1)}{30(q-1)} - \frac{q^2}{15} (q^{12} - 1) &\text{if $3 \mid \ind b$}   \\
  \end{cases}
\]
for $2 \mid r$, $4 \nmid r$, and
\begin{align*}
  P_{30} &- \frac{q^{15} (q^{14} - 1)}{30(q-1)} \\
  &= 
  \begin{cases}
    -\frac{1}{15} \bigl( 2(\omega_{15}^\frac{15r}{2} + \bar{\omega}_{15}^\frac{15r}{2}) + 3 \bigr) q^{14} \\
    \hspace{3em}- \frac{q^5 (q^4 - 1)}{10(q-1)} - \frac{q^3}{6} (q+1) + \frac{q^2}{15} (6q^2 + 5) 
      &\text{if $15 \mid \ind b$,} \\
    \frac{1}{30} \bigl( \omega_{15}^\frac{15r}{2} + \bar{\omega}_{15}^\frac{15r}{2} - 1 \bigr) q^{14} 
      - \frac{q^4 (q^5 - 1)}{10(q-1)} 
      &\text{if $\ind b \in C_3^{15}$,} \\
    \frac{1}{30} \bigl( 2(\omega_{15}^\frac{15r}{2} + \bar{\omega}_{15}^\frac{15r}{2}) - 3 \bigr) q^{14} 
      - \frac{q^2(q^3 - 1)}{6(q-1)} 
      &\text{if $\ind b \in C_5^{15}$,} \\
    \frac{1}{15} \bigl( \omega_{15}^{\frac{15r}{2} + 1} + \bar{\omega}_{15}^{\frac{15r}{2} + 1} + 1 \bigr) q^{14} 
      &\text{if $\ind b \in C_c^{15}$,} \\
    \frac{1}{15} \bigl( \omega_{15}^{\frac{15r}{2} - 1} + \bar{\omega}_{15}^{\frac{15r}{2} - 1} + 1 \bigr) q^{14} 
      &\text{if $\ind b \in C_{-c}^{15}$,}
  \end{cases}
\end{align*}
for $4 \mid r$, where $\omega_{15} = -(1 + \sqrt{-15})/4$. 

The index $2$ case \ref{it:Npqnoro} holds for $m = 21$. 
The Gauss sums $F_6 (\chi) = -2(3 + c\sqrt{-7})$ and $F_6 (\chi^3) = 2(3 + c\sqrt{-7}) = -F_6 (\chi)$, where $c \in \{ 1, -1 \}$, 
are computed in \cite[Example 11]{RinThe}. 
Then
\[
  P_{21} - \frac{q^{20} - 1}{21(q-1)} = -\frac{q^6 + q^2 - 2}{21(q-1)}
\]
for $2 \nmid r$, $3 \nmid r$,
\[
  P_{21} - \frac{q^{20} - 1}{21(q-1)} = 
  \begin{cases}
    -\frac{1}{21} \bigl( q + 1 \mp (q^9 - 1) \sqrt{q} \bigr)                        &\text{if $3 \nmid \ind b$,} \\
    -\frac{1}{21} \bigl( \frac{3q^6 + q^2 - 4}{q-1} \pm 2 (q^9 - 1) \sqrt{q} \bigr) &\text{if $3 \mid \ind b$}   \\
  \end{cases}
\]
for $2 \mid r$, $3 \nmid r$,
\begin{align*}
  &P_{21} - \frac{q^{20} - 1}{21 (q-1)} = \\
  &\begin{cases}
    - \bigl( \frac{q^6 + 7q^2 - 8}{21(q-1)} + \frac{\sqrt{q^5}}{7} \bigl( (\omega_7^{7r} + \bar{\omega}_7^{7r}) q^7 
       + (\omega_7^\frac{7r}{3} + \bar{\omega}_7^\frac{7r}{3}) \bigr) \bigr)
      &\text{if $7 \mid \ind b$,}       \\
    -\frac{q^6 - 1}{21(q-1)} + \frac{\sqrt{2q^5}}{21} \bigl( (\omega_7^{7r - 1} + \bar{\omega}_7^{7r - 1}) q^7 
      - (\omega_7^{\frac{7r}{3} + 1} + \bar{\omega}_7^{\frac{7r}{3} + 1}) \bigr) \hspace{-3em}\\
      &\text{if $\ind b \in C_c^7$,}    \\
    -\frac{q^6 - 1}{21(q-1)} + \frac{\sqrt{2q^5}}{21} \bigl( (\omega_7^{7r + 1} + \bar{\omega}_7^{7r + 1}) q^7 
      - (\omega_7^{\frac{7r}{3} - 1} + \bar{\omega}_7^{\frac{7r}{3} - 1}) \bigr) \hspace{-3em}\\
      &\text{if $\ind b \in C_{-c}^7$}     \\
  \end{cases}
\end{align*}
for $2 \nmid r$, $3 \mid r$, and
\begin{align*}
  &P_{21} - \frac{q^{20} - 1}{21 (q-1)} = \\
  &\begin{cases}
    -\bigl( \frac{1}{21} \bigl( 3(2 \pm 1) (\omega_{21}^\frac{7r}{2} + \bar{\omega}_{21}^\frac{7r}{2}) \pm 2 q^7 \bigr) \sqrt{q^5} \\
      \hspace{1.5em}+ \frac{1}{7} \bigl( \tfrac{q^6 - 1}{q-1} - 3(\omega_7^\frac{7r}{3} + \bar{\omega}_7^\frac{7r}{3}) \sqrt{q^5} \bigr) 
      + \frac{1}{3} (1 \mp \sqrt{q})^2 \bigr)                           &\text{if $21 \mid \ind b$,} \\
    \frac{1}{21} \bigl( 3(1 \mp 1) (\omega_{21}^\frac{7r}{2} + \bar{\omega}_{21}^\frac{7r}{2}) \pm q^7 \bigr) \sqrt{q^5} 
      - \frac{1}{3} (q + 1 \pm \sqrt{q})                                          &\text{for $C_7^{21}$,} \\
    \frac{\sqrt{q^5}}{21} \bigl( 2 \Real(\omega_{21}^\frac{7r}{2} (1 + \sqrt{-7})) 
      \pm \Real(\omega_{21}^\frac{7r}{2} (1 - \sqrt{-7})) \mp 2q^7 \bigr) \hspace{-4pt} \\
      \hspace{1.5em}- \frac{1}{7} \bigl( \tfrac{q^6 - 1}{q-1} + (\omega_7^{\frac{7r}{3} - 1} + \bar{\omega}_7^{\frac{7r}{3} - 1}) 
      \sqrt{2q^5} \bigr)                                                                         &\text{for $C_{3c}^{21}$,} \\
    \frac{\sqrt{q^5}}{21} \bigl( 2 \Real(\omega_{21}^\frac{7r}{2} (1 - \sqrt{-7})) 
      \pm \Real(\omega_{21}^\frac{7r}{2} (1 + \sqrt{-7})) \mp 2q^7 \bigr) \hspace{-4pt} \\
      \hspace{1.5em}- \frac{1}{7} \bigl( \tfrac{q^6 - 1}{q-1} + (\omega_7^{\frac{7r}{3} + 1} + \bar{\omega}_7^{\frac{7r}{3} + 1}) 
      \sqrt{2q^5} \bigr)                                                                        &\text{for $C_{-3c}^{21}$,} \\
    \frac{\sqrt{q^5}}{21} \bigl( -\Real(\omega_{21}^\frac{7r}{2} (1 - \sqrt{-7})) \pm \Real(\omega_{21}^\frac{7r}{2} (1 + \sqrt{-7})) 
      \pm q^7 \bigr) \hspace{-4pt} &\text{for $C_c^{21}$,} \\
    \frac{\sqrt{q^5}}{21} \bigl( -\Real(\omega_{21}^\frac{7r}{2} (1 + \sqrt{-7})) \pm \Real(\omega_{21}^\frac{7r}{2} (1 - \sqrt{-7})) 
      \pm q^7 \bigr) \hspace{-4pt} &\text{for $C_{-c}^{21}$}
  \end{cases}
\end{align*}
for $6 \mid r$, where $C_i^{21}$ indicates the cyclotomic coset that $\ind b$ belongs to and $\pm = (-1)^\frac{r}{2}$. 
In addition, $\omega_7 = (1 + \sqrt{-7})/\sqrt{8}$ is as in $P_7$ and $\omega_{21} = 3 + \sqrt{-7}$.

The irreducible polynomial generator in \cite{polgen} can be used to cross-check our formulae for small values of $q$ and $m$. 
It lists every irreducible polynomial over $\Fq$, $q \le 8$, of a given degree $m$ if there are at most $1000$ such polynomials. 
We can use \cite{polgen} to list every irreducible polynomial of degree $m \le 13$, $6$, $3$ over $\F_2$, $\F_4$, $\F_8$, respectively. 
One can then pick the polynomials with $a=0$ from this list. 
On the other hand, the formulae of this subsection give the number of these polynomials. 
For example, let $m=3$. 
If $q=2$, $8$, then \eqref{eq:P3} gives $P_3 = \tfrac{q^2 - 1}{3 \cdot (q-1)}$ for every $b$. 
The $P_3$ equals to $1$ if $q=2$, and to $3$ if $q=8$. 
If $q=4$ then \eqref{eq:P3} gives $P_3 = \tfrac{q^2 - 1}{3 \cdot (q-1)} - \tfrac{1}{3} \sqrt{q} = 1$ for $3 \nmid \ind b$ 
($b \ne 1$; $2$ values), and $P_3 = \tfrac{q^2 - 1}{3 \cdot (q-1)} + \tfrac{2}{3} \sqrt{q} = 3$ for $3 \mid \ind b$ ($b=1$; $1$ value). 
The other values in Table \ref{tbl:smallpol} are obtained similarly. 
Our results agree with those obtained using \cite{polgen}.

\begin{table}[!htb]
\centering
\caption{The number of the irreducible polynomials with $a=0$ and $b$ fixed for small $q$ and $m$.}
\label{tbl:smallpol}
\begin{tabular}{ccccccccccccc}
\hline
     $q$      &                           \multicolumn{12}{c}{$m$}                              \\
\cline{2-13}
              & $2$ & $3$ & $4$ & $5$  & $6$  & $7$ & $8$  & $9$  & $10$ & $11$ & $12$  & $13$  \\
\hline
     $2$      & $0$ & $1$ & $1$ & $3$  & $4$  & $9$ & $14$ & $28$ & $48$ & $93$ & $165$ & $315$ \\
  $4$, $b=1$  & $0$ & $3$ & $4$ & $17$ & $48$                                                   \\
$4$, $b \ne1$ & $0$ & $1$ & $4$ & $17$ & $56$                                                   \\
     $8$      & $0$ & $3$                                                                       \\
\hline
\end{tabular}
\end{table}


\end{document}